\documentclass[12pt]{article}

\usepackage{a4,graphicx,amsmath,amsfonts,amssymb,bbm,amsthm}
\usepackage{color,url}
\usepackage{mathtools}

\newcommand{\real}{\mathbbm{R}}

\newcommand{\nat}{\mathbb{N}}

\newtheorem{theorem}{Theorem}
\newtheorem{corollary}{Corollary}
\newtheorem{lemma}{Lemma}

\newtheorem{remark}{Remark}

\begin{document}
\begin{center}
  {\bf \Large Balanced truncation for model order reduction \\[0.5ex]
    of linear dynamical systems \\[1.5ex]
    with quadratic outputs}

  \vspace{5mm}

  \vspace{0.5cm}   

{\large Roland~Pulch$\mbox{}^1$\footnote{corresponding author} and Akil Narayan$\mbox{}^2$}

\vspace{0.1cm}

\begin{small}

{$\mbox{}^1$Institute of Mathematics and Computer Science, 
University of Greifswald, \\
Walther-Rathenau-Str.~47, 
D-17489 Greifswald, Germany.}\\
Email: {\tt roland.pulch@uni-greifswald.de}

{$\mbox{}^2$Scientific Computing and Imaging Institute,
  and Department of Mathematics, \\
  University of Utah, 
  72~Central Campus Dr., Salt Lake City, UT 84112,
  United States.}\\
  Email: {\tt akil@sci.utah.edu}

\end{small}
  
\end{center}

\bigskip\bigskip


\begin{center}
{Abstract}

\begin{tabular}{p{13cm}}

  We investigate model order reduction (MOR) for linear dynamical systems,
  where a quadratic output is defined as a quantity of interest.
  The system can be transformed into a linear dynamical system with
  many linear outputs.
  MOR is feasible by the method of balanced truncation,
  but suffers from the large number of outputs in approximate methods.
  To ameliorate this shortcoming we derive an equivalent quadratic-bilinear
  system with a single linear output and analyze the properties of this system.
  We examine MOR for this system via the technique of balanced truncation,
  which requires a stabilization of the system.
  Therein, the solution of two quadratic Lyapunov equations is traced back
  to the solution of just two linear Lyapunov equations.
  We present numerical results for several test examples
  comparing the two MOR approaches.

  \bigskip


\bigskip


\bigskip

\mbox{}
\end{tabular}
\end{center}

\thispagestyle{plain}

\clearpage

\markright{R.~Pulch, A.~Narayan: Balanced truncation for quadratic outputs}


\section{Introduction}

A mathematical modeling of physical or industrial applications often
yields dynamical systems.
Automatic model generation typically generates dynamical systems
of enormous dimension,
and thus repeated transient simulations may become too costly.
In this situation methods of model order reduction (MOR) are required to decrease
the computational complexity of the system.
Efficient MOR techniques already exist for linear dynamical systems,
see~\cite{antoulas,benner,freund,gugercin,schilders}.
MOR by balanced truncation or moment matching, for example,
is based on an approximation of the input-output mapping,
which can be described by a transfer function in the frequency domain.
In contrast, the design of efficient and accurate algorithms for MOR of
{\em nonlinear} dynamical systems is still a challenging task.

In this paper we investigate linear dynamical systems in the form of
ordinary differential equations.
However, the output quantity of interest for our system is the time-dependent trajectory of a quadratic function of the state.
Differential equations with quadratic outputs arise in mechanical applications like mass-spring-damper systems, see~\cite{depken}, for example.
Maxwell's equations with quadratic outputs were examined
using finite element methods in~\cite{hammerschmidt}. 
In stochastic models, the variance of a quantity of interest represents a quadratic function, cf.~\cite{haasdonk}. 
This nonlinear (quadratic) dependence of outputs on inputs precludes the direct applicability of transfer function methods for linear time-invariant systems.

There is some existing literature that addresses MOR for linear dynamical systems with quadratic outputs: Van~Beeumen et al.~\cite{beeumen-meerbergen,beeumen2012} consider
a linear dynamical system with a single quadratic output
in the frequency domain.
Those authors transform this system into an equivalent form with multiple linear outputs. 
This approach is also applicable to our case of a linear dynamical system with
a single quadratic output in the time domain, but often suffers from a very large number of outputs and hence is computationally expensive.

The theory of differential balancing was designed for general nonlinear
dynamical systems, see~\cite{kawano-scherpen-2017,scherpen}.
This concept also applies to our case of linear dynamical systems
with quadratic outputs.
Concerning the two involved Gramian matrices, one is constant and the other
is state-dependent.
Thus the numerical solution of the defining matrix-valued equations becomes
too costly for an efficient MOR.

The strategy we employ is to derive a quadratic-bilinear dynamical system,
whose single linear output coincides with the quadratic output of
the original linear dynamical system.
This allows us to leverage MOR methods for quadratic-bilinear systems,
e.g.~\cite{ahmad,benner-breiten}.
We find the balanced truncation technique introduced by
Benner and Goyal~\cite{benner-goyal} particularly effective.
Therein, quadratic Lyapunov equations must be solved,
whose solutions are constant Gramian matrices.
We perform a stabilization of our quadratic-bilinear system to ensure the solvability of the Lyapunov equations.
When analyzing the structure of the resulting quadratic-bilinear system,
MOR can be accomplished by solving just two linear Lyapunov equations.
In particular, approximate methods like the alternating direction
implicit (ADI) scheme can be used to solve the linear Lyapunov equations
numerically \cite{li-white,lu-wachspress,mmess,penzl,penzl-sisc}. 

The paper is organized as follows.  We derive both the linear dynamical system with multiple outputs and the quadratic-bilinear system with single output in
Section~\ref{sec:problem-def}.
The balanced truncation technique is applied to the quadratic-bilinear
system in Section~\ref{sec:truncation}.
Therein, we analyze the structure and deduce an efficient solution
of the Lyapunov equations.
Section~\ref{sec:example} presents results of numerical computations
for three test examples, and both MOR approaches are compared
with respect to accuracy and computation work.


\section{Problem definition: Linear dynamical \\
  systems with quadratic outputs}
\label{sec:problem-def}

Let a linear time-invariant system be given with a quadratic output
in the form
\begin{align} \label{dynamicalsystem}
\begin{split}  
\dot{x}(t) & = A x(t) + B u(t) \\[1ex] 
y(t) & = x(t)^\top M x(t) . 
\end{split} 
\end{align}
The state variables $x : [0,t_{\rm end}] \rightarrow \real^n$ are determined
by the matrix $A \in \real^{n \times n}$,
the matrix $B \in \real^{n \times n_{\rm in}}$
and given inputs $u : [0,t_{\rm end}] \rightarrow \real^{n_{\rm in}}$.
We suppose that the number of inputs is small ($n_{\rm in} \ll n$).
We assume that the system is asymptotically stable, i.e.,
all eigenvalues of~$A$ exhibit a (strictly) negative real part.
An initial value problem is defined by
\begin{equation} \label{ivp}
  x(0) = x_0
\end{equation}
with a predetermined $x_0 \in \real^n$.
The quantity of interest~$y : [0,t_{\rm end}] \rightarrow \real$
represents a quadratic output defined by the symmetric
matrix $M \in \real^{n \times n}$.
The system~(\ref{dynamicalsystem}) is multiple-input-single-output (MISO).

\begin{remark} \label{remark:symm}
Note that, if $T \in \real^{n \times n}$ is any matrix, then 
$$  x^\top T x = x^\top \left( {\textstyle \frac{1}{2}} (T + T^\top ) \right) x
  \qquad \mbox{for all} \;\; x \in \real^n . $$
Thus we can assume that the matrix $M$ in \eqref{dynamicalsystem} is symmetric (since the relation above implies it can be replaced by its symmetric part).
\end{remark}

We assume a situation with large dimension $n$.
(See Section~\ref{sec:example} for examples of such systems.)
It is in this regime when MOR algorithms are advantageous.
The aim of MOR is to construct a (linear or quadratic) dynamical system
of a much lower dimension $r \ll n$, and an associated output $\bar{y}$, such that the difference $y - \bar{y}$ should be sufficiently small for all relevant times.

\subsection{Transformation to linear dynamical systems \\ with multiple outputs}
For comparison, we adopt the approach in \cite{beeumen2012},
where a linear dynamical system with a single quadratic output is transformed into a linear dynamical system with multiple (linear) outputs.
We can consider again the system~(\ref{dynamicalsystem}) but with a
new output of interest $z$:
\begin{align} \label{simo-system}
\begin{split}  
\dot{x}(t) & = A x(t) + B u(t), \\[1ex] 
z(t) & = L^\top x(t).
\end{split} 
\end{align}
The matrix $L \in \real^{n \times m}$ is any matrix satisfying $M = L L^\top$
with $m = {\rm rank}(M) \le n$ and $z : [0,t_{\rm end}] \rightarrow \real^m$.
The initial values are as in~(\ref{ivp}).
The system~(\ref{simo-system}) is multiple-input-multiple-output (MIMO).
The relationship between $L$ and $M$ immediately yields a relationship between the output $z$ in \eqref{simo-system} and the output $y$ in \eqref{dynamicalsystem}:
\begin{equation} \label{output-y-z}
  y(t) = x(t)^\top M x(t) = (L^\top x(t))^\top (L^\top x(t)) =
  z(t)^\top z(t) = \| z(t) \|_2^2,
\end{equation}
where $\|\cdot\|_2$ is the Euclidean norm.

The matrix $L$ can be identified via an arbitrary symmetric decomposition $M = L L^\top$. For example, if $M$ is positive semi-definite, then the pivoted Cholesky factorization \cite{higham} provides one method of symmetric decomposition.
Negative semi-definite matrices~$M$ can be replaced by $-M$
and the negative Euclidean norm from~(\ref{output-y-z}) is used.
In the case of indefinite symmetric matrices, an eigen-decomposition
yields
\begin{equation} \label{Mdecomposition}
  M = S (D_+ - D_-) S^\top =
  S D_+ S^\top - S D_- S^\top \eqqcolon M_+ - M_-
\end{equation}
for an orthogonal matrix $S \in \real^{n \times n}$
and diagonal matrices $D_+ , D_-$ including the positive eigenvalues
and the modulus of the negative eigenvalues, respectively. The matrices $M_+$ and $M_-$ are both positive semi-definite so that the decomposition~(\ref{Mdecomposition}) can be used to construct matrices $L_{\pm}$ satisfying:
\begin{align*}
  M_{\pm} &= L_{\pm} L^{\top}_{\pm}, & L_{\pm} &\in \real^{n \times m_{\pm}}, & \mathrm{rank}(M) = m_+ + m_-.
\end{align*}
Now it holds that $M = L_+ L_+^\top - L_- L_-^\top$.
Thus we arrange the linear dynamical system in the form
\begin{align*}
  \begin{split}
    \dot{x}(t) & = A x(t) + B u(t) \\[1ex]
    \begin{pmatrix} z_+(t) \\ z_-(t) \\ \end{pmatrix} & =
    \begin{pmatrix} {L_+}^\top \\ {L_-}^\top \\ \end{pmatrix} x(t) , 
  \end{split}
\end{align*}
so that 
  $y(t) = \| z_+ (t) \|_2^2 - \| z_- (t) \|_2^2$
relates the quadratic output of~\eqref{dynamicalsystem} to the linear outputs $z_{\pm}$.

If the rank of the matrix~$M$ is low, then just a few outputs
arise in the system~(\ref{simo-system}) and efficient MOR methods are available to reduce the linear dynamical system.
However the situation becomes more difficult when the matrix $M$ has a large rank since MOR approaches often suffer limitations from the presence of a large number of outputs, with the extreme case given by a matrix~$M$ of full rank.
Some MOR methods achieve accuracy only when $\mathrm{rank}(M) \ll n$,
see~\cite[p.~231]{beeumen2012}, for example,
so that little to no computational benefit is gained in this case.
We will see for one of our capstone examples in Section \ref{sec:example} a realistic example where $M = I$, which presents a significant disadvantage to directly applying MOR to \eqref{simo-system}.

\subsection{Transformation to quadratic-bilinear systems }
\label{sec:trafo-quadratic}
This section presents a strategy to overcome the limitations described at the end of the previous section. The idea is to construct a dynamical system that includes the
quadratic quantity of interest $y$ as an additional state variable.
Differentiation of the quadratic output from~(\ref{dynamicalsystem})
yields $\dot{y} = 2 \dot{x}^\top M x$, and using \eqref{dynamicalsystem}
results in
$$ \dot{y} = 2 \left( A x + B u \right)^\top M x =
x^\top (2 A^\top M) x + u^\top (2 B^\top M) x . $$
This reveals that $\dot{y}$ as a function of the state $x$ can be expressed as the sum of a quadratic term and a bilinear term. 
The matrix $2 A^\top M$ is not symmetric in general, but
Remark~\ref{remark:symm} shows that we can replace this matrix by
\begin{equation} \label{symmetricmatrix}
  S := A^\top M + M^\top A = A^\top M + M A .
\end{equation}
Furthermore, let $b_1,b_2,\ldots,b_{n_{\rm in}}$ be the columns of
the matrix~$B$.

By constructing a new vector $\tilde{x} \coloneqq (x^\top, \; y)^\top \in \real^{n+1}$, we obtain
\begin{align} \label{qdr-bil-system}
\begin{split}  
  \dot{\tilde{x}}(t) & = \tilde{A} \tilde{x}(t)
  + \tilde{B} u(t)
  + \displaystyle \sum_{j=1}^{n_{\rm in}} u_j(t) \tilde{N}_j \tilde{x}(t)
  + \tilde{H} ( \tilde{x}(t) \otimes \tilde{x}(t) ) \\
  \tilde{y}(t) & = \tilde{c}^\top \tilde{x}(t), 
\end{split} 
\end{align}
with
$$ \tilde{A} = \begin{pmatrix} A & 0 \\ 0 & 0 \\ \end{pmatrix} , \qquad
\tilde{B} = \begin{pmatrix} B \\ 0 \\ \end{pmatrix} , \qquad
\tilde{c}^\top = (0,\ldots,0,1) , \qquad 
\tilde{N}_j = \begin{pmatrix} 0 & 0 \\ 2 b_j^\top M & 0 \\ \end{pmatrix} $$
for $j=1,\ldots,n_{\rm in}$.
The quadratic term involves the Kronecker product
$\tilde{x} \otimes \tilde{x} \in \real^{(n+1)^2}$.
The matrix $\tilde{H} \in \real^{(n+1) \times (n+1)^2}$ can be defined in terms of
the columns $s_1,\ldots,s_n \in \real^n$ of the symmetric matrix~$S$
from~(\ref{symmetricmatrix}):
\begin{equation} \label{structure-H}
\tilde{H} =
\begin{pmatrix}
  0_{n \times n} & 0_n & 0_{n \times n} & 0_n & \cdots & 0_{n \times n} & 0_n & 0_{n \times n} & 0_n \\
  s_1^\top & 0 & s_2^\top & 0 & \cdots & s_n^\top & 0 & 0_{1 \times n} & 0 \\
\end{pmatrix} ,
\end{equation}
where only the last row is occupied.
Thus the quadratic-bilinear system~(\ref{qdr-bil-system}) produces just a single linear output, which is identical to the $(n+1)$th state variable.  The initial values~(\ref{ivp}) imply that we must augment \eqref{qdr-bil-system} with 
$$ \tilde{x} (0) =
\begin{pmatrix} x_0 \\ x_0^\top M x_0 \\ \end{pmatrix} . 
$$
The quadratic output~$y$ of~(\ref{dynamicalsystem}) and the
linear output~$\tilde{y}$ of~(\ref{qdr-bil-system}) coincide.

Our approach is to apply MOR methods for quadratic-bilinear systems, where an advantage is that the system~(\ref{qdr-bil-system}) is MISO
with a low number of inputs by assumption.
We will show that the specific structure of \eqref{qdr-bil-system} 
allows for an efficient MOR computation.
Note that our system~(\ref{qdr-bil-system}) is well-defined for arbitrary
output matrices~$M$, i.e., including indefinite matrices.

\subsubsection{Relationship to tensors}
The system~\eqref{qdr-bil-system} can be related to more standard or classical definitions of quadra\-tic dynamical systems.
A quadratic dynamical system is defined by a three-dim\-en\-sional
tensor~$\mathcal{H} \in \real^{k \times k \times k}$.
In the system~(\ref{qdr-bil-system}),
the matrix~$\tilde{H} = \mathcal{H}^{(1)}$ represents the 1-matriciza\-tion
of this tensor.
Since the matrix~$S$ in~(\ref{symmetricmatrix}) is symmetric,
its rows and columns coincide.
Consequently, the 2-matricization $\mathcal{H}^{(2)}$ and the
3-matricization $\mathcal{H}^{(3)}$ are identical, i.e., 
$$ \mathcal{H}^{(2)} = \mathcal{H}^{(3)} = 
\begin{pmatrix}
  0_{n \times n} & s_1 & 0_{n \times n} & s_2 & \cdots & 0_{n \times n} & s_n & 0_{n \times n} & 0_n \\
  0_{1 \times n} & 0 & 0_{1 \times n} & 0 & \cdots & 0_{1 \times n} & 0 & 0_{1 \times n} & 0 \\
\end{pmatrix} . $$
It follows that the underlying tensor is symmetric,
which allows for advantageous algebraic manipulations.
However, an unsymmetric tensor can always be symmetrized.
The definition of the tensor matricizations and symmetric tensors
can be found in~\cite{benner-breiten}, for example.

\subsubsection{Simplification of the bilinear term}
\label{sec:simpl-bilinear}
The matrices of the bilinear part in~(\ref{qdr-bil-system})
become $\tilde{N}_j = 0$ when $b_j^\top M=0$ for each~$j$.
This simplification can happen quite often in practice, in particular when output-relevant state variables have equations that do not include the input.
Let $\mathcal{I}_{b_j} \subset \{ 1,\ldots,n \}$ be the support indices of $b_j$,
and let $\mathcal{I}_M \subset \{ 1,\ldots,n \}$ denote the 
subset of state variables, which are involved in the quadratic output.
It holds that
\begin{equation} \label{N-is-zero}
  \mathcal{I}_{b_j} \cap \mathcal{I}_M = \emptyset
  \qquad \Rightarrow \qquad
  b_j^\top M = 0 .
\end{equation}
If the premise of~\eqref{N-is-zero} is satisfied for all $j=1,\ldots,n_{\rm in}$
then the bilinear term in~\eqref{qdr-bil-system} vanishes.

\subsubsection{Stabilization}
The linear term $\tilde{A} \tilde{x}$ of the system~(\ref{qdr-bil-system}) is just Lyapunov stable and not asymptotically stable, because the matrix~$\tilde{A}$ has
a zero eigenvalue that is simple. However, asymptotic stability is mandatory in some MOR methods. Hence we stabilize via a tunable parameter $\varepsilon > 0$, resulting in the system 
\begin{align} \label{system-stable}
\begin{split}  
  \dot{\tilde{x}}(t) & = \tilde{A}(\varepsilon) \tilde{x}(t)
  + \tilde{B} u(t)
  + \displaystyle \sum_{j=1}^{n_{\rm in}} u_j(t) \tilde{N}_j \tilde{x}(t)
  + \tilde{H} ( \tilde{x}(t) \otimes \tilde{x}(t) ) \\
  \tilde{y}(t) & = \tilde{c}^\top \tilde{x}(t) 
\end{split} 
\end{align}
with the matrix
\begin{equation} \label{A-eps}
  \tilde{A}(\varepsilon) =
  \begin{pmatrix} A & 0 \\ 0 & -\varepsilon \\ \end{pmatrix}.
\end{equation}
The asymptotic stability of the system~(\ref{system-stable}), and
thus the $\varepsilon > 0$ assumption, is crucial in our balanced truncation method,
which will be demonstrated in Section~\ref{sec:gramians}.


\section{Balanced Truncation}
\label{sec:truncation}
Our goal is now to compute an MOR of the stabilized quadratic-bilinear
system~\eqref{system-stable}.
To achieve this we apply a method of balanced truncation.
The ultimate accomplishment of balanced truncation is determination of transformation matrices $\tilde{T}_r, \tilde{T}_l \in \real^{(n+1) \times (n+1)}$. Truncations of these matrices to size $\real^{(n+1) \times r}$, with $r \ll n$, are subsequently used to transform the system \eqref{system-stable} of size $n+1$ into an approximating system of size $r$. Since $r \ll n$, this therefore accomplishes MOR. Our ultimate identification of these transformation matrices is given by \eqref{square-transformation} in Section \ref{sec:decompositions}, and an appropriate truncation is prescribed in Section \ref{sec:mor}. An overview of the entire algorithm we propose is demonstrated in Section \ref{ssec:algorithm}.

  The first step in a balanced truncation approach is computation of two Gramian matrices, and our approach to compute these matrices is given in Section \ref{sec:gramians}. However, we first take a short detour in Section \ref{sec:diff-balancing} to illustrate why a related but alternative MOR approach, namely the approach differential balancing, includes substantial challenges and is not attractive for our class of problems.

\subsection{Differential balancing}
\label{sec:diff-balancing}
For comparison against balanced truncation, we examine the concept of differential balancing
from~\cite{kawano-scherpen-2017,scherpen} applied to our class of nonlinear problems.
The purpose of this section is to illustrate that differential balancing presents significant challenges.
In Section \ref{sec:gramians} we will see that such challenges do not arise using our approach of balanced truncation.

\subsubsection{Gramian matrices}
We consider a general nonlinear dynamical system
\begin{align} \label{nonlinear-dynamical}
  \begin{split}
  \dot{x}(t) & = f(t,x(t)) + g(t,x(t)) u(t) \\[1ex]
  y(t) & = h(t,x(t)) . 
  \end{split}
\end{align}
The Gramian matrices of the system~(\ref{nonlinear-dynamical})
may depend on time and/or the state space.
For a matrix-valued function $A(t,x) = (a_{ij}) \in \real^{n \times n}$ and
a vector-valued function $f(t,x) \in \real^n$, we use the notation
$\delta_f (A) = ( \frac{\partial a_{ij}}{\partial t} +
\frac{\partial a_{ij}}{\partial x} f )$.
The reachability Gramian~$P \in \real^{n \times n}$ associated to \eqref{nonlinear-dynamical} satisfies the equations
\begin{align} \label{reach-diff}
  \begin{split}
    - \delta_f (P(t,x)) + \frac{\partial f(t,x)}{\partial x} P(t,x)
    + P(t,x) {\frac{\partial f(t,x)}{\partial x}}^\top 
    + g(t,x) g(t,x)^\top & = 0 \\[1ex]
    - \delta_{g_j} (P(t,x)) + \frac{\partial g_j(t,x)}{\partial x} P(t,x)
    + P(t,x) { \frac{\partial g_j(t,x)}{\partial x} }^\top & = 0 
  \end{split}
\end{align}
for $j=1,\ldots,n_{\rm in}$.
The observability Gramian~$Q$ is the solution of the equations
\begin{equation} \label{obser-diff}
  - \delta_f (Q(t,x)) + \textstyle
  { \frac{\partial f(t,x)}{\partial x} }^\top Q(t,x)
  + Q(t,x) \frac{\partial f(t,x)}{\partial x} +
    { \frac{\partial h(t,x)}{\partial x} }^\top
    \frac{\partial h(t,x)}{\partial x} = 0 .
\end{equation}
The existence and uniqueness of solutions to the above equations is not guaranteed \textit{a priori}.
  
\subsubsection{Linear dynamical system with quadratic output}
Only the output is nonlinear in the dynamical system~(\ref{dynamicalsystem}).
The following definitions of the functions in \eqref{nonlinear-dynamical}
yield the special case \eqref{dynamicalsystem}:
$$ f(t,x) = A x , \qquad g(t,x) = B , \qquad h(t,x) = x^\top M x $$
without an explicit time-dependence.
On the one hand,
let $P$ be the constant matrix solving the linear Lyapunov equation
\begin{equation} \label{lyap-linear-P}
  A P + P A^\top + B B^\top = 0
\end{equation}
of the linear case.
It follows that the matrix~$P$ satisfies all equations~(\ref{reach-diff}).
Hence the differential balancing coincides with the linear concept.
On the other hand, the equations~(\ref{obser-diff}) become
$$  - \delta_f (Q(x)) + A^\top Q(x) + Q(x) A + 4 M x x^\top M = 0 $$ 
assuming a time-invariant solution.
This problem is much more complicated than the linear case,
since the solution still depends on the state space and
partial differential equations emerge.

In~\cite[p.~3302]{kawano-scherpen-2017}, the technique of
generalized differential balancing (gDB) was introduced to achieve
an MOR with a reasonable computational work.
This approach requires an input term~$B(t)u(t)$ and an output
$y(t) = C(t)x(t)$, where the matrices do not depend on the state space.
Hence gDB cannot be directly applied to the dynamical system~(\ref{dynamicalsystem})
due to the nonlinear output.

\subsubsection{Quadratic-bilinear dynamical system}
The function definitions that cast the general system \eqref{nonlinear-dynamical} into the special quad\-ra\-tic-bilinear system \eqref{qdr-bil-system} are:
$$ f(t,\tilde{x}) = \begin{pmatrix} Ax \\ x^\top S x \\ \end{pmatrix} , \qquad
g(t,\tilde{x}) = \begin{pmatrix} B \\ 2 x^\top MB \\ \end{pmatrix} , \qquad
h(t,\tilde{x}) = \tilde{c}^\top \tilde{x} . $$
Again there is no explicit time-dependence.
The Jacobian matrices with respect to the state space read as
$$ \frac{\partial f}{\partial \tilde{x}} =
\begin{pmatrix} A & 0 \\ 2 x^\top S & 0 \\ \end{pmatrix}
\qquad \mbox{and} \qquad
 \frac{\partial g_j}{\partial \tilde{x}} =
\begin{pmatrix} 0 & 0 \\ 2 b_j^\top M & 0 \\ \end{pmatrix} $$
for $j=1,\ldots,n_{\rm in}$.
The first part of the equations~(\ref{reach-diff}) becomes
\begin{equation} \label{reach-diff2}
  \begin{array}{rcl}
  - \delta_f (\tilde{P}(x)) +
\begin{pmatrix} A & 0 \\ 2 x^\top S & 0 \\ \end{pmatrix} \tilde{P}(x)
+ \tilde{P}(x) \begin{pmatrix} A^\top & 2Sx \\ 0 & 0 \\ \end{pmatrix} & & \\[2ex]
+ \begin{pmatrix} BB^\top & 2BB^\top Mx \\
  2 x^\top M BB^\top & 4 x^\top M B B^\top M x \\ \end{pmatrix} & = & 0 \\
  \end{array}
\end{equation}
with a state-dependent solution~$\tilde{P}(x)$.
Let
\begin{equation} \label{ansatz-Px}
  \tilde{P}(x) =
  \begin{pmatrix} P & v(x) \\ v(x)^\top & w(x) \\ \end{pmatrix}
\end{equation}
with $P \in \real^{n \times n}$ satisfying the
Lyapunov equation~(\ref{lyap-linear-P}) and
$v(x) \in \real^n$, $w(x) \in \real$.
The ansatz~(\ref{ansatz-Px}) solves the left upper minor
of the matrix-valued system~(\ref{reach-diff2}).
However, the other minors remain state-dependent.
The other parts of the equations~(\ref{reach-diff}) are simpler,
because the Jacobian matrices $\frac{\partial g_j}{\partial \tilde{x}}$
are constant, but still represent differential equations.
The equations~(\ref{obser-diff}) simplify to
$$ - \delta_f (\tilde{Q}(x)) +
\begin{pmatrix} A^\top & 2Sx \\ 0 & 0 \\ \end{pmatrix} \tilde{Q}(x)
+ \tilde{Q}(x) \begin{pmatrix} A & 0 \\ 2 x^\top S & 0 \\ \end{pmatrix} 
+ \begin{pmatrix} 0 & 0 \\ 0 & 1 \\ \end{pmatrix} = 0 $$
with a state-dependent solution~$\tilde{Q}(x)$.
The lower right entry shows
$(\delta_f (\tilde{Q}))_{\tilde{n},\tilde{n}} = 1$ with $\tilde{n}=n+1$,
which implies that the matrix~$\tilde{Q}$ is not constant. 

In conclusion, the reachability Gramian $P$ as well as observability Gramian
$Q$ require the solution of matrix-valued differential equations,
which is computationally expensive.
Our approach in the following section derives constant (time- and state-independent) Gramian matrices.

If it holds that $N_j \neq 0$ for some~$j$,
then the gDB technique cannot be applied to the quadratic-bilinear
system~(\ref{qdr-bil-system}), because the input part does not
exhibit the simple form $B(t) u(t)$.
If it holds that $N_j = 0$ for all~$j$, then gDB is feasible. 
The generalized differential Gramians are non-unique solutions
of matrix inequalities, see~\cite[p.~3302]{kawano-scherpen-2017},
which become constant matrices in this special case.
Yet the matrix inequalities depend on the state variables and
have to be satisfied for all states. 
Thus the complexity is still high in comparison to
linear or quadratic Lyapunov equations with constant coefficients.

\subsection{Gramians of quadratic-bilinear system}
\label{sec:gramians}
The reachability and observability Gramian matrices are defined for
general quadratic-bilinear systems with multiple inputs
and multiple outputs (MIMO) in~\cite{benner-goyal}.
The theorems on Gramian matrices require a 
quadratic-bilinear system with an asymptotically stable linear part.
Moreover, in \cite{benner-goyal} the existence of the Gramian matrices
is assumed,
i.e., the solvability of quadratic Lyapunov equations alone does not
imply that a solution represents a Gramian,
cf. Theorems 3.1 and 3.2 of \cite{benner-goyal}.
However, the existence of these Gramians can be used to give bounds on observability and controllability energy functionals, similar to the linear case, cf. Theorems 4.1 and 4.2 of \cite{benner-goyal}. Furthermore, truncations of these Gramian matrices also provide observability and controllability estimates for reduced systems. Therefore, computation of these matrices, and subsequent truncations of them, are of fundamental importance for our MOR strategy.

The reachability Gramian $\tilde{P} \in \real^{(n+1) \times (n+1)}$ associated
to the stabilized system~\eqref{system-stable} is the solution of the
quadratic Lyapunov equation
\begin{equation} \label{eq-reachability}
  \tilde{A}(\varepsilon) \tilde{P} + \tilde{P} \tilde{A}(\varepsilon)^\top +
  \tilde{B} \tilde{B}^\top +
  \tilde{H} (\tilde{P} \otimes \tilde{P} ) \tilde{H}^\top +
  \displaystyle \sum_{j=1}^{n_{\rm in}} \tilde{N}_j \tilde{P} \tilde{N}_j^\top 
  = 0 .
\end{equation}
The observability Gramian $\tilde{Q} \in \real^{(n+1) \times (n+1)}$ satisfies
the quadratic Lyapunov equation
\begin{equation} \label{eq-observability}
  \tilde{A}(\varepsilon)^\top \tilde{Q} + \tilde{Q} \tilde{A}(\varepsilon) +
  \tilde{c} \tilde{c}^\top +
  \tilde{H}^{(2)} (\tilde{P} \otimes \tilde{Q} ) (\tilde{H}^{(2)})^\top +
  \displaystyle \sum_{j=1}^{n_{\rm in}} \tilde{N}_j^\top \tilde{Q} \tilde{N}_j 
  = 0 .
\end{equation}
If the reachability Gramian~$\tilde{P}$ is given, then the Lyapunov
equations~(\ref{eq-observability}) represent a linear system for
the unknown entries of the observability Gramian $\tilde{Q}$.
Both Gramian matrices are symmetric and positive semi-definite
since $\varepsilon > 0$.

The following lemma compiles relations that are used to evaluate the
terms in the Lyapunov equations.

\begin{lemma} \label{lemma:formulas}
  Let $\tilde{P},\tilde{Q} \in \real^{(n+1) \times (n+1)}$ be partitioned into
  $$ \tilde{P} = \begin{pmatrix} P & 0 \\ 0 & p' \\ \end{pmatrix}
  \qquad \mbox{and} \qquad
  \tilde{Q} = \begin{pmatrix} Q & 0 \\ 0 & q' \\ \end{pmatrix} $$
  with symmetric matrices $P,Q \in \real^{n \times n}$ and $p',q' \in \real$.
  It follows that
  \begin{itemize}
  \item[i)]
  $\tilde{N}_j \tilde{P} \tilde{N}_j^\top = 4
    \begin{pmatrix} 0 & 0 \\ 0 & b_j^\top M P M b_j \\ \end{pmatrix}$
    \quad for each~$j$,
  \item[ii)]
    $\displaystyle \sum_{j=1}^{n_{\rm in}} \tilde{N}_j^\top \tilde{Q} \tilde{N}_j =
    4 q' \begin{pmatrix} M BB^\top M & 0 \\ 0 & 0 \\ \end{pmatrix}$,
  \item[iii)]
  $\tilde{H} (\tilde{P} \otimes \tilde{P} ) \tilde{H}^\top =
  {\rm tr} ((PS)^2) \; \tilde{c} \, \tilde{c}^\top$, 
  \item[iv)]
  $\tilde{H}^{(2)} (\tilde{P} \otimes \tilde{Q} ) (\tilde{H}^{(2)})^\top =
  q' \begin{pmatrix} S P S & 0 \\ 0 & 0 \\ \end{pmatrix}$. 
  \end{itemize}
\end{lemma}

\begin{proof}
\leavevmode\newline
i) We calculate directly
$$ \tilde{N}_j \tilde{P} \tilde{N}_j^\top =
\begin{pmatrix} 0 & 0 \\ 2 b_j^\top M & 0 \\ \end{pmatrix}
\begin{pmatrix} P & 0 \\ 0 & p' \\ \end{pmatrix}
\begin{pmatrix} 0 & 2 M b_j \\ 0 & 0 \\ \end{pmatrix} =
\begin{pmatrix} 0 & 0 \\ 0 & 4 b_j^\top M P M b_j \\ \end{pmatrix} . $$
ii) Likewise, it holds that
$$ \tilde{N}_j^\top \tilde{Q} \tilde{N}_j =
\begin{pmatrix} 0 & 2Mb_j \\ 0 & 0 \\ \end{pmatrix}
\begin{pmatrix} Q & 0 \\ 0 & q' \\ \end{pmatrix}
\begin{pmatrix} 0 & 0 \\ 2 b^\top M & 0 \\ \end{pmatrix} =
\begin{pmatrix} 4q' (Mb_j)(b_j^\top M) & 0 \\ 0 & 0 \\ \end{pmatrix} $$
for~$j=1,\ldots,n_{\rm in}$.
The sum over~$j$ yields the formula.

iii) The structure~\eqref{structure-H} of $\tilde{H}$ implies that the matrix
$\tilde{H} (\tilde{P} \otimes \tilde{P} ) \tilde{H}^\top$
has only one non-zero entry in the position $(n+1,n+1)$.
The vector $\tilde{c}$ is the $(n+1)$th unit vector.
Hence we obtain
$\tilde{H} (\tilde{P} \otimes \tilde{P} ) \tilde{H}^\top =
\gamma \tilde{c} \, \tilde{c}^\top$
with a scalar $\gamma$ to be determined.

Let $t_{ij}$ for $i,j=1,\ldots,n$ be the entries of the
non-symmetric matrix $PS$.
It holds that
$$ \gamma = \sum_{i,j=1}^n p_{ji} s_j^\top P s_i =
\sum_{i,j,k,\ell = 1}^n p_{ij} s_{jk} p_{k\ell} s_{i\ell} =
\sum_{i,k=1}^n t_{ik} t_{ki} . $$
Furthermore, we obtain the entries
$$ ((PS)^2)_{ij} = \sum_{k=1}^n t_{ik} t_{kj}
\quad \mbox{and thus} \quad
((PS)^2)_{ii} = \sum_{k=1}^n t_{ik} t_{ki} $$
for $i,j=1,\ldots,n$.
The sum over~$i$ yields the trace.

iv) We define the symmetric matrix $\tilde{S} \in \real^{(n+1)\times(n+1)}$ by
$$ \tilde{S} = \begin{pmatrix} S & 0 \\ 0 & 0 \\ \end{pmatrix} . $$
It holds that
$\tilde{H}^{(2)} = \tilde{S} \otimes \tilde{c}^\top \in \real^{(n+1) \times (n+1)^2}$,
since $\tilde{c}$ is the $(n+1)$th unit vector.
The rule for matrix multiplications with the Kronecker product yields
$$ \tilde{H}^{(2)} ( \tilde{P} \otimes \tilde{Q} ) (\tilde{H}^{(2)})^\top =
(\tilde{S} \otimes \tilde{c}^\top) ( \tilde{P} \otimes \tilde{Q} )
(\tilde{S} \otimes \tilde{c}) =
(\tilde{S} \tilde{P} \tilde{S}) \otimes
(\tilde{c}^\top \tilde{Q} \tilde{c}) =
q' \tilde{S} \tilde{P} \tilde{S} . $$
The definition of the matrix $\tilde{S}$ implies
$$ \tilde{S} \tilde{P} \tilde{S} = 
\begin{pmatrix} SPS & 0 \\ 0 & 0 \\ \end{pmatrix} , $$
which shows the statement.
\end{proof}

\medskip

The left-hand sides of (i)-(iv) are terms in the Lyapunov equations \eqref{eq-reachability} and \eqref{eq-observability}.  Direct evaluation of terms in the Lyapunov equations can be expensive, but Lemma~\ref{lemma:formulas} indicates that this effort can be significantly reduced because of the special structure of the system \eqref{system-stable} under consideration. 
The terms~(i) and~(ii) together require mainly two matrix-vector
multiplications with the columns of~$B$.
The term~(iv) is obtained by two matrix-matrix multiplications
for $SPS$.
One additional matrix-matrix multiplication yields $(PS)^2$ and
thus the term~(iii).

The next result characterizes the solutions of the quadratic
Lyapunov equations.
\begin{theorem}\label{thm:gramian-computation}
  A reachability Gramian~$\tilde{P}$
  satisfying~(\ref{eq-reachability}) 
  is given by
  \begin{equation} \label{ansatzP}
  \tilde{P} =
  \begin{pmatrix} P & 0 \\ 0 & p' \end{pmatrix} ,
  \end{equation}
  where $P$ solves the linear Lyapunov equation~(\ref{lyap-linear-P}),
  i.e., $AP + PA^\top + BB^\top = 0$, and
  \begin{equation} \label{pprime}
    p' = \textstyle{\frac{1}{2\varepsilon}} p''
    \quad \mbox{with} \quad
    p'' = 
    {\rm tr}((PS)^2)
    + 4 \displaystyle \sum_{j=1}^{n_{\rm in}} b_j^\top M P M b_j
    \ge 0 .
  \end{equation}
  An observability Gramian~$\tilde{Q}$
  satisfying~(\ref{eq-observability})
  is given by
  \begin{equation} \label{ansatzQ}
  \tilde{Q} = \frac{1}{2\varepsilon}
  \begin{pmatrix} Q & 0 \\ 0 & 1 \end{pmatrix},
  \end{equation}
  where $Q$ satisfies the linear Lyapunov equations
  \begin{equation} \label{lyap-linear-Q}
    A^\top Q + Q A + SPS + 4 M B B^\top M = 0 .
  \end{equation}
\end{theorem}

\begin{proof}
\leavevmode\newline
Inserting the ansatz~(\ref{ansatzP}) into the quadratic
Lyapunov equations~(\ref{eq-reachability}) yields the linear
Lyapunov equations~(\ref{lyap-linear-P}) for the first part.
Due to Lemma~\ref{lemma:formulas} (i) and (iii),
the second part becomes
$$ - 2 \varepsilon p' + {\rm tr} ((PS)^2) +
4 \displaystyle \sum_{j=1}^{n_{\rm in}} b_j^\top M P M b_j = 0 , $$
which uniquely defines the scalars~$p'$ and~$p''$, respectively.

Now the ansatz~(\ref{ansatzQ}) is inserted into the
Lyapunov equations~(\ref{eq-observability}).
The second part is fulfilled immediately
due to $-2 \varepsilon \cdot \frac{1}{2\varepsilon} + 1 = 0$. 
Lemma~\ref{lemma:formulas} (ii) and (iv) yield the
linear Lyapunov equations
$$ \textstyle A^\top \frac{1}{2\varepsilon} Q +
A\frac{1}{2\varepsilon} Q  + \frac{1}{2\varepsilon}  SPS +
4 \frac{1}{2\varepsilon} M B B^\top M = 0 $$
as the first part.
A multiplication by the factor $2\varepsilon$ results in the
Lyapunov equations~(\ref{lyap-linear-Q}).

Finally, we show the lower bound $p'' \ge 0$ in~(\ref{pprime}).
The solution~$P$ of the Lyapunov equation~(\ref{lyap-linear-P})
is always symmetric and positive semi-definite.
The matrix~$S$ is just symmetric.
It follows that $b_j^\top M P M b_j \ge 0$ for all $j=1,\ldots,n_{\rm in}$.
Using the matrix square root $P = P^{\frac{1}{2}} P^{\frac{1}{2}}$,
we obtain
$$ {\rm tr}((PS)^2) = {\rm tr}( P^{\frac{1}{2}} P^{\frac{1}{2}} SPS ) =
{\rm tr} ( P^{\frac{1}{2}} SPS P^{\frac{1}{2}} ) =
{\rm tr} ( (SP^{\frac{1}{2}})^\top P (S P^{\frac{1}{2}}) ) $$
due to a property of the trace.
Obviously, the matrix $(SP^{\frac{1}{2}})^\top P (S P^{\frac{1}{2}})$
is symmetric and positive semi-definite again.
Hence its trace is non-negative.
\end{proof}

\medskip

We have shown the existence of symmetric positive semi-definite solutions
satisfying the quadratic Lyapunov equations~(\ref{eq-reachability})
and (\ref{eq-observability}).
The proof of Theorem~\ref{thm:gramian-computation} demonstrates that
the matrices~(\ref{ansatzP}) and~(\ref{ansatzQ}) are the unique
solutions of the Lyapunov equations in the set of
block-diagonal matrices of the used form.
We just assume that the solutions are also unique in the set of
all matrices.
All these property have to be assumed in the case of general
quadratic-bilinear systems, cf.~\cite[p.~13]{benner-goyal-gugercin}.

Theorem \ref{thm:gramian-computation} reveals the explicit dependence of the reachability and observability Gramians $\tilde{P}$ and $\tilde{Q}$, respectively, on the stabilization parameter $\varepsilon$. In particular, the matrices $P,Q$ satisfying the linear Lyapunov equations
(\ref{lyap-linear-P}),(\ref{lyap-linear-Q}) are independent of
the stabilization parameter~$\varepsilon$. The observability Gramian~(\ref{ansatzQ}) is directly proportional to $\frac{1}{\varepsilon}$.
This $1/\varepsilon$ dependence suggests that the Lyapunov equations cannot
be solved if we set $\varepsilon = 0$.
We codify this fact below.
  \begin{corollary}
    If $\varepsilon = 0$ and $P^{\frac{1}{2}} S P^{\frac{1}{2}} \neq 0$,
    then the Lyapunov equation \eqref{eq-reachability}
    does not have a solution.
    If $\varepsilon = 0$, then the Lyapunov equation \eqref{eq-observability}
    does not have a solution.
  \end{corollary}
  \begin{proof}
    \leavevmode\newline
    Let $\varepsilon = 0$ in the matrix~(\ref{A-eps}).
    We investigate the component $(n+1,n+1)$ in each quadratic Lyapunov equation.
    In the Lyapunov equation~(\ref{eq-reachability}),
    this component yields $p'' = 0$ with $p''$ defined in~(\ref{pprime}). 
    We showed that $p'' \ge 0$ in the proof of
    Theorem~\ref{thm:gramian-computation}. 
    A necessary condition for $p''=0$ is $P^{\frac{1}{2}} SPS P^{\frac{1}{2}}=0$,
    which is excluded due to $(P^{\frac{1}{2}} S P^{\frac{1}{2}})^2 \neq 0$.
    In the Lyapunov equation~(\ref{eq-observability}),
    the left-hand side of this component becomes~$1$ due to
    Lemma~\ref{lemma:formulas} (ii) and (iv).
    This contradicts the fact that this component on the right-hand side of \eqref{eq-observability} must be 0.
  \end{proof}


\subsection{Balancing the system}
\label{sec:decompositions}
In order to perform balanced truncation, we require symmetric decompositions of the two Gramian matrices.
Let $P = L_P L_P^\top$ and $Q = L_Q L_Q^\top$ be the Cholesky decompositions
of the solutions of the linear Lyapunov equations
(\ref{lyap-linear-P}) and (\ref{lyap-linear-Q}), respectively.
From Theorem~\ref{thm:gramian-computation},
we obtain factorizations of the reachability Gramian
$\tilde{P} = \tilde{L}_P \tilde{L}_P^\top$
and the observability Gramian
$\tilde{Q} = \tilde{L}_Q \tilde{L}_Q^\top$ by
$$ \tilde{L}_P =
\begin{pmatrix} L_P & 0 \\ 0 & \sqrt{\frac{p''}{2\varepsilon}} \\
\end{pmatrix}
\qquad \mbox{and} \qquad
\tilde{L}_Q = \frac{1}{\sqrt{2\varepsilon}}
\begin{pmatrix} L_Q & 0 \\ 0 & 1 \\
\end{pmatrix} . $$
We remark that one need not actually form the Gramian matrices
in order to compute low-rank Cholesky factors \cite{Kue16}.
In the method of balanced truncation, we require the
singular value decomposition (SVD)
\begin{equation} \label{svd}
  \tilde{L}_Q^\top \tilde{L}_P =
  \tilde{U} \tilde{\Sigma} \tilde{V}^\top .
\end{equation}
In our case, the matrix on the left-hand side reads as
$$ \tilde{L}_Q^\top \tilde{L}_P =  \frac{1}{\sqrt{2\varepsilon}}
\begin{pmatrix} L_Q^\top L_P & 0 \\ 0 & \sqrt{\frac{p''}{2\varepsilon}} \\
\end{pmatrix} . $$
We use the SVD
\begin{equation} \label{svd-small}
  L_Q^\top L_P = U \Sigma V^\top ,
\end{equation}
where the diagonal matrix $\Sigma = {\rm diag}(\sigma_1,\ldots,\sigma_n)$
includes the singular values in \textit{ascending} order
$\sigma_1 \le \sigma_2 \le \cdots \le \sigma_n$.
The SVD~(\ref{svd-small}) is independent of the
stabilization parameter~$\varepsilon$.
Thus the SVD component matrices in (\ref{svd}) are given by
\begin{equation}\label{eq:SVD-tilde}
\tilde{U} = \begin{pmatrix} U & 0 \\ 0 & 1 \\ \end{pmatrix} , \qquad
\tilde{V}^\top = \begin{pmatrix} V^\top & 0 \\ 0 & 1 \\ \end{pmatrix} , \qquad
\tilde{\Sigma} = \frac{1}{\sqrt{2\varepsilon}}
\begin{pmatrix} \Sigma & 0 \\ 0 & \sqrt{\frac{p''}{2\varepsilon}} \\
\end{pmatrix} .
\end{equation}
If the stabilization parameter $\varepsilon$ is sufficiently small,
then the maximum singular value is
$\sqrt{p''} / (2\varepsilon)$.
However, the associated singular vector is independent of
$p''$ and $\varepsilon$.
The singular values of the quadratic-bilinear system~(\ref{qdr-bil-system})
are
\begin{equation} \label{sv-order}
  \textstyle \frac{1}{\sqrt{2\varepsilon}}
  \left( \sigma_1 , \sigma_2 , \ldots , \sigma_n ,
  \sqrt{\frac{p''}{2\varepsilon}} \right) .
\end{equation}
These real numbers represent the analogue of the Hankel singular values
in the case of linear dynamical systems.
The two transformation matrices, which achieve a balanced system,
result to
\begin{equation} \label{square-transformation}
\begin{array}{l}
\tilde{T}_{\rm l} =
\tilde{L}_Q \tilde{U} \tilde{\Sigma}^{-\frac{1}{2}} =
\frac{1}{\sqrt[4]{2\varepsilon}}
\begin{pmatrix}
  L_Q U \Sigma^{-\frac{1}{2}} & 0 \\
  0 & \sqrt[4]{\frac{2\varepsilon}{p''}} \\
\end{pmatrix} =
\begin{pmatrix}
  \frac{1}{\sqrt[4]{2\varepsilon}} L_Q U \Sigma^{-\frac{1}{2}} & 0 \\
  0 & {p''}^{-\frac{1}{4}} \\
\end{pmatrix} , \\[4ex]
\tilde{T}_{\rm r} =
\tilde{L}_P \tilde{V} \tilde{\Sigma}^{-\frac{1}{2}} =
\sqrt[4]{2\varepsilon}
\begin{pmatrix}
  L_P V \Sigma^{-\frac{1}{2}} & 0 \\
  0 & \sqrt[4]{\frac{p''}{2\varepsilon}} \\
\end{pmatrix} =
\begin{pmatrix}
  \textstyle{\sqrt[4]{2\varepsilon}} L_P V \Sigma^{-\frac{1}{2}} & 0 \\
  0 & {p''}^{\frac{1}{4}} \\
\end{pmatrix}
\end{array}
\end{equation}
with $p''$ from~(\ref{pprime}).
It holds that
$\tilde{T}_{\rm l}^\top \tilde{T}_{\rm r} = I$
with the identity matrix~$I$.
For $\varepsilon \rightarrow 0$ some parts of the matrices converge
to zero and the other parts tend to infinity, hence the limits do not exist.
However, the balanced system can be written in a form, which allows
for further interpretations.

\begin{lemma}\label{lemma:balanced-system}
  The balanced system of dimension $n+1$ reads as
  \begin{align} \label{system-balanced}
  \begin{split}  
    \dot{\bar{x}}(t) & = \bar{A}(\varepsilon) \bar{x}(t)
    + \bar{B}(\varepsilon) u(t)
    + \displaystyle \sum_{j=1}^{n_{\rm in}} u_j(t) \bar{N}_j(\varepsilon) \bar{x}(t)
    + \bar{H}(\varepsilon) ( \bar{x}(t) \otimes \bar{x}(t) ) \\
    \bar{y}(t) & = \bar{c}^\top \bar{x}(t) 
  \end{split} 
  \end{align}
  with
  $$ \bar{A}(\varepsilon) =
  \begin{pmatrix} \bar{A}' & 0 \\ 0 & - \varepsilon \\ \end{pmatrix}, \qquad
  \bar{B}(\varepsilon) =  \frac{1}{\sqrt[4]{2\varepsilon}}
  \begin{pmatrix} \bar{B}' \\ 0 \\ \end{pmatrix}, \;\; $$
  $$ \bar{N}_j(\varepsilon) = \sqrt[4]{2\varepsilon}
  \begin{pmatrix} 0 & 0 \\ \bar{N}_j' & 0 \\ \end{pmatrix} , \qquad
  \bar{H}(\varepsilon) = \sqrt{2\varepsilon} \bar{H}' $$
  and $\bar{A}',\bar{B}',\bar{N}_j',\bar{H}',\bar{c}$
  independent of~$\varepsilon$.
\end{lemma}

\begin{proof}
  \leavevmode\newline
  It holds that
  $\bar{c}^\top = \tilde{c}^\top \tilde{T}_{\rm r} =
  (0,\ldots,0,\sqrt[4]{p''})^\top$.
  We obtain
  $\bar{A}(\varepsilon) =
  \tilde{T}_{\rm l}^\top \tilde{A}(\varepsilon) \tilde{T}_{\rm r}$
  and
  $\bar{B}(\varepsilon) = \tilde{T}_{\rm l}^\top \tilde{B}$.
  It follows that
  $$ \bar{A}' = \left( \Sigma^{-\frac{1}{2}} \right)^{\top}
  U^\top L_Q^\top A L_P V \Sigma^{-\frac{1}{2}}
  \qquad \mbox{and} \qquad
  \bar{B}' = \begin{pmatrix}
    \left( \Sigma^{-\frac{1}{2}} \right)^{\top} U^\top L_Q^\top B \\ 0 \\
  \end{pmatrix} . $$
  The matrices of the bilinear part become
  $$ \bar{N}_j(\varepsilon) =
  \tilde{T}_{\rm l}^\top \tilde{N}_j(\varepsilon) \tilde{T}_{\rm r} =
  \sqrt[4]{2\varepsilon}
  \begin{pmatrix}
    0 & 0 \\
    2 {p''}^{-\frac{1}{4}} b_j^\top M L_P V  \Sigma^{-\frac{1}{2}} & 0 \\
  \end{pmatrix}
  \quad \mbox{for} \;\; j=1,\ldots,n_{\rm in} . $$
  The quadratic part exhibits the structure~(\ref{structure-H}).
  It hold that
  $\bar{H}(\varepsilon) = \tilde{T}_{\rm l}^\top \tilde{H}
  (\tilde{T}_{\rm r} \otimes \tilde{T}_{\rm r})$. 
  We obtain $\tilde{T}_{\rm l}^\top \tilde{H} =  {p''}^{-\frac{1}{4}} \tilde{H}$
  due to the structure of $\tilde{T}_{\rm l}$.
  Let
  $$ \tilde{T}_{\rm r}' =
  \begin{pmatrix}
    L_P V \Sigma^{-\frac{1}{2}} & 0 \\
    0 & {p''}^{\frac{1}{4}} \\
  \end{pmatrix} . $$
  It follows that
  $\bar{H}(\varepsilon) =  \sqrt{2\varepsilon} {p''}^{-\frac{1}{4}} \tilde{H}
  (\tilde{T}_{\rm r}' \otimes \tilde{T}_{\rm r}')$,
  where a factor $\sqrt[4]{2\varepsilon}$ comes from
  each~$\tilde{T}_{\rm r}$
  and thus
  $\sqrt[4]{2\varepsilon} \cdot \sqrt[4]{2\varepsilon} = \sqrt{2\varepsilon}$.
  \hfill \mbox{}
\end{proof}

Lemma~\ref{lemma:balanced-system} shows that the differential equations
of the balanced system can be decoupled into the parts 
\begin{align} \label{balanced-parts}
  \begin{split}  
    \dot{\bar{x}}^*(t) & = \bar{A}' \bar{x}^*(t)
    + \bar{B}' \left( \frac{1}{\sqrt[4]{2\varepsilon}} u(t) \right) \\[1ex]
    \dot{\bar{x}}_{n+1}(t) & =
    - \varepsilon \bar{x}_{n+1}(t)
    + \displaystyle \sum_{j=1}^{n_{\rm in}} u_j(t) \bar{N}_j'
    ( \sqrt[4]{2\varepsilon} \bar{x}^*(t) )
    + \bar{H}'' ( (\sqrt[4]{2\varepsilon} \bar{x}^*(t)) \otimes
                (\sqrt[4]{2\varepsilon} \bar{x}^*(t) )) 
  \end{split}
\end{align}
with a row vector~$\bar{H}''$.
Initial values are transformed via
$$ \bar{x}(0) = \tilde{T}_{\rm r}^{-1} \tilde{x}(0) =
\begin{pmatrix}
  \frac{1}{\sqrt[4]{2\varepsilon}} \Sigma^{\frac{1}{2}} V^{\top} L_P^{-1} x_0 \\
       {p''}^{-\frac{1}{4}} x_0^\top M x_0 \\
\end{pmatrix} . $$
It follows that the solution $\bar{x}^*$ is directly proportional to
$\frac{1}{\sqrt[4]{2\varepsilon}}$
(change in input signal as well as change in initial values).
For $0 < \varepsilon < \frac{1}{2}$, this amplification is canceled by
multiplication of $\bar{x}^*$ with the factor $\sqrt[4]{2\varepsilon}$
in the last equation of~(\ref{balanced-parts}).
It follows that the system
\begin{align} \label{balanced-parts2}
  \begin{split}  
    \dot{\bar{x}}^*(t) & = \bar{A}' \bar{x}^*(t)
    + \bar{B}' u(t) \\
    \dot{\bar{x}}_{n+1}(t) & =
    -\varepsilon \bar{x}_{n+1}(t)
    + \displaystyle \sum_{j=1}^{n_{\rm in}} u_j(t) \bar{N}_j' \bar{x}^*(t)
    + \bar{H}''
    ( \bar{x}^*(t) \otimes \bar{x}^*(t) ) 
  \end{split}
\end{align}
with initial values $\bar{x}^*(0) = \Sigma^{\frac{1}{2}} V^{\top} L_P^{-1} x_0$
exhibits the same solution $\bar{x}_{n+1}$ as in
the system~(\ref{balanced-parts}).
Now~(\ref{balanced-parts2}) includes the parameter~$\varepsilon$
only in the scalar term of the last equation.
We may arrange $\varepsilon \rightarrow 0$ to eliminate the
dependence on the stabilization parameter completely.


\subsection{Reduced-order model}
\label{sec:mor}
The concepts of reachability and observability allow one to devise an MOR strategy: State variables components that require a large energy to achieve (reach)
or generate a low energy in the output (observation) should be truncated.
In the balanced systems, a state variable is hard to reach if and only if
it produces a low output energy.
In contrast to the linear case,
error estimates are not available in the case
of quadratic-bilinear systems yet.

Given a reduced-order dimension $r \in \nat$, we assume that the stabilization parameter $\varepsilon>0$ is chosen
sufficiently small such that
\begin{equation} \label{sv-large}
  \sigma_{n+1} := \textstyle \sqrt{\frac{p''}{2\varepsilon}} > \sigma_{n-r} .
\end{equation}
Hence $\sigma_{n+1}$ belongs to the set of the $r$ dominant
singular values.
We partition the SVD~(\ref{svd}) into
\begin{equation} \label{svd-partition}
  \tilde{L}_Q^\top \tilde{L}_P =
  \begin{pmatrix} \tilde{U}_1 & \tilde{U}_2 \\ \end{pmatrix}
  \begin{pmatrix}
  \tilde{\Sigma}_1 & 0 \\ 0 & \tilde{\Sigma}_2 \\
  \end{pmatrix}
  \begin{pmatrix} \tilde{V}_1^\top \\ \tilde{V}_2^\top \\ \end{pmatrix}
\end{equation}
with $\tilde{\Sigma}_2 \in \real^{r \times r}$,
$\tilde{U}_2 , \tilde{V}_2 \in \real^{(n+1) \times r}$.
The associated projection matrices
$\tilde{T}_{\rm l / r} \in \real^{(n+1) \times r}$
read as
\begin{equation}\label{eq:Tlr-truncation}
  \tilde{T}_{\rm l} =
\tilde{L}_Q \tilde{U}_2 \tilde{\Sigma}_2^{-\frac{1}{2}}
\qquad \mbox{and} \qquad
\tilde{T}_{\rm r} =
\tilde{L}_P \tilde{V}_2 \tilde{\Sigma}_2^{-\frac{1}{2}}. 
\end{equation}
Due to the ascending order of the singular values and the
condition~(\ref{sv-large}), the MOR truncates state variables, which are
both hard to reach and difficult to observe.

The reduced-order model (ROM) of the quadratic-bilinear
system~(\ref{system-stable}) becomes
\begin{align} \label{system-reduced}
\begin{split}  
  \dot{\bar{x}}(t) & = \bar{A} \bar{x}(t)
  + \bar{B} u(t)
  + \displaystyle \sum_{j=1}^{n_{\rm in}} u_j(t) \bar{N}_j \bar{x}(t)
  + \bar{H} ( \bar{x}(t) \otimes \bar{x}(t) ) \\ 
  \bar{y}(t) & = \bar{c}^\top \bar{x}(t) 
\end{split} 
\end{align}
with the solution $\bar{x} : [0,t_{\rm end}] \rightarrow \real^r$
and the downsized matrices
\begin{equation} \label{reduced-matrices}
  \begin{array}{l}
  \bar{A} = \tilde{T}_{\rm l}^\top \tilde{A} \tilde{T}_{\rm r} ,
  \qquad
  \bar{B} = \tilde{T}_{\rm l}^\top \tilde{B} , \qquad
  \bar{c}^\top = \tilde{c}^\top \tilde{T}_{\rm r} , \\[1ex]
  \bar{N}_j = \tilde{T}_{\rm l}^\top \tilde{N}_j \tilde{T}_{\rm r} , \qquad
  \bar{H} = \tilde{T}_{\rm l}^\top \tilde{H}
  ( \tilde{T}_{\rm r} \otimes \tilde{T}_{\rm r} ) . \\
  \end{array}
\end{equation}
The output vector is
$$ \bar{c}^\top = \left( 0,\ldots,0,\sqrt[4]{ p'' } \right) . $$
Thus the output is just a multiple of the final state variable as
in the quadratic-bilinear system~(\ref{system-stable}).
Initial values $\bar{x}(0)$ have to be determined from~\eqref{ivp}.
The balanced truncation method preserves the local asymptotic stability
of the equilibrium $\tilde{x}_{\rm eq}=0$ in autonomous
systems~(\ref{qdr-bil-system}) with $u \equiv 0$,
see~\cite{benner-goyal}.

The computational effort for the matrices~$\bar{N}_j$
in~(\ref{reduced-matrices}) is negligible because only the last row
of~$\tilde{N}_j$ is non-zero. 
The computation of the matrix~$\bar{H} \in \real^{r \times r^2}$ represents
the most expensive part in the projections~\eqref{reduced-matrices}.
In~\cite[p.~245]{benner-breiten}, an algorithm is outlined to
construct~$\bar{H}$ without calculating the Kronecker product
$\tilde{T}_{\rm r} \otimes \tilde{T}_{\rm r}$ explicitly.
In our case, the effort becomes even lower, since the matrix~$\bar{H}$
exhibits the structure~\eqref{structure-H} of~$H$.
The entries of the matrix
\begin{equation} \label{T-S-T}
  \bar{S} = \tilde{T}_{\rm r}^\top S \tilde{T}_{\rm r} \in \real^{r \times r},
\end{equation}
with the symmetric matrix~$S$ from~(\ref{symmetricmatrix}),
yield the last row of the intermediate matrix
$\tilde{H} ( \tilde{T}_{\rm r} \otimes \tilde{T}_{\rm r} )$,
whereas the other rows are zero.
Thus the effort mainly consists in the computation of
matrix-matrix products in~\eqref{T-S-T}.

\begin{remark}
We obtain 
\begin{equation} \label{barA}
  \bar{A} =
  \begin{pmatrix}
    \bar{A}^* & 0 \\ 0 & - \varepsilon \\
  \end{pmatrix}
\end{equation}
with a matrix $\bar{A}^* \in \real^{(r-1) \times (r-1)}$ independent of
$\varepsilon$.
This allows us in principle to consider the limit
$\varepsilon \rightarrow 0$ in the matrix~(\ref{barA}).
\end{remark}

The ROM~(\ref{system-reduced}) exhibits the same structure as
the quadratic-bilinear system~(\ref{system-stable}) in the nonlinear terms.

\begin{theorem} \label{thm:rom}
  Let
  $\bar{x} = (\bar{x}_1,\ldots,\bar{x}_r)^\top$ and
  $\bar{x}^* = (\bar{x}_1,\ldots,\bar{x}_{r-1})^\top$.
The reduced system~(\ref{system-reduced}) has
the equivalent form
\begin{align} \label{system-reduced2}
  \begin{split}
  \dot{\bar{x}}^*(t) & = \bar{A}^* \bar{x}^*(t) + \bar{B}^* u(t) \\
  \dot{\bar{x}}_{r}(t)  & = - \varepsilon \bar{x}_r(t)
  + \displaystyle \sum_{j=1}^{n_{\rm in}} u_j(t) \bar{N}_j^* \bar{x}^*(t)
  + \bar{H}^* ( \bar{x}^*(t) \otimes \bar{x}^*(t) ) \\
  \bar{y}(t) & = \sqrt[4]{ p'' } \; \bar{x}_r(t) 
  \end{split}
\end{align}
with $p''$ from~\eqref{pprime}, the matrix~$\bar{A}^*$ from~(\ref{barA}),
and a modified matrix~$\bar{B}^*$ and modified row vectors $\bar{N}_j^*,\bar{H}^*$.
\end{theorem}
The proof is straightforward.
The structure of the system~(\ref{system-reduced2}) implies two
benefits for solving initial value problems in comparison to
general quadratic-bilinear systems:
\begin{enumerate}
\item
  If an implicit time integration scheme is used, then nonlinear systems
  of algebraic equations can be avoided and only linear systems
  have to be solved.
\item
  Since the output is just a constant multiple of a single state variable,
  an adaptive time step size selection can be performed by a
  local error control of this state variable only.
\end{enumerate}
Moreover, the matrix~$\bar{H}^*$ in~(\ref{system-reduced2}) has the
structure~(\ref{structure-H}). 
We collect its non-zero entries in a symmetric matrix
$\bar{S}^* \in \real^{(r-1) \times (r-1)}$.
Now let $\varepsilon = 0$.
In view of~(\ref{symmetricmatrix}), we consider the symmetric solution
$\bar{M}^*$ of the Lyapunov equation
$$ \bar{A}^{*\top} \bar{M}^* + \bar{M}^* \bar{A}^* = \bar{S}^* . $$
If the linear dynamical system with quadratic output
\begin{align} \label{reduced-linear}
  \begin{split}
  \dot{\bar{x}}^*(t) & = \bar{A}^* \bar{x}^*(t) + \bar{B}^* u(t) \\[1ex]
  \tilde{y}(t) & = \bar{x}^*(t)^\top \bar{M}^* \bar{x}^*(t) . \\
  \end{split}
\end{align}
is transformed into a quadratic-bilinear system as in
Section~\ref{sec:trafo-quadratic}, then the quad\-ra\-tic term
of system~(\ref{system-reduced2}) exactly appears.
However, the bilinear part becomes different, even if all
matrices~$\tilde{N}_j$ (and thus $\bar{N}_j^*$) were zero.
Hence the structure of the original system~(\ref{dynamicalsystem})
cannot be retrieved by this MOR.
Only in the autonomous case ($u \equiv 0$),
the dynamical systems~(\ref{system-reduced2}) and~(\ref{reduced-linear})
are equivalent ($\bar{y} = \sqrt[4]{ p'' } \tilde{y}$).

Lemma~\ref{lemma:balanced-system} implies that the stabilization
parameter~$\varepsilon$ influences only the scalar term of the last
differential equation in the ROM~(\ref{system-reduced2}).
If we choose $\varepsilon=0$ in~(\ref{barA}) or, equivalently,
in the scalar term of~(\ref{system-reduced2}),
then the output becomes independent of the parameter.
The value~$\varepsilon$ just has to be sufficiently small such that
the last singular value in~(\ref{sv-order}) belongs to the dominant
singular values used to determine the ROM.

\subsection{Low-rank approximations}
\label{sec:lowrank}
An MOR for the linear dynamical system~(\ref{dynamicalsystem})
with quadratic output can be performed by using the
linear dynamical system~(\ref{simo-system}) with multiple outputs
or the quadratic-bilinear system~(\ref{qdr-bil-system}) with single output.
Two criteria determine the efficiency of the approaches
in balanced truncation:
\begin{enumerate}
\item
  The decay of the singular values.
  A faster decay typically allows for a sufficiently accurate
  ROM of a lower dimension.
\item
  The computational effort to construct an ROM.
\end{enumerate}
Numerical computations indicate that the rate of decay is similar
for the singular values in test examples.
Thus an advantage in the quadratic-bilinear system formulation can be achieved
only by decreasing the computational effort.

The main part of the computational work for the balanced truncation technique
consists in the solution of the linear Lyapunov equations
(\ref{lyap-linear-P}) and (\ref{lyap-linear-Q}).
A general linear Lyapunov equation reads as
\begin{equation} \label{lyapunov-general}
  A G + G A^\top + F = 0
\end{equation}
for the unknown matrix~$G \in \real^{n \times n}$
with a given symmetric positive semi-definite matrix~$F \in \real^{n \times n}$.
If we apply direct methods of linear algebra to
solve~(\ref{lyapunov-general}),
then the computational complexity is $O(n^3)$ and
nearly independent of the rank of~$F$.
Consequently, we could solve the linear dynamical system~(\ref{simo-system})
including many outputs as well, where $F$ exhibits a high rank.

Alternatively, approximate methods yield low-rank factorizations of the
solution of the Lyapunov equation~(\ref{lyapunov-general}).
Efficient algorithms are achieved by iterations based on
the alternating direction implicit (ADI) technique,
see~\cite{li-white,lu-wachspress}.
A low-rank approximation reads as
$G \approx Z_G Z_G^\top$ with $Z_G \in \real^{n \times k}$
for some $k \ll n$.
An ADI technique requires a symmetric factorization
$F \approx Z_F Z_F^\top$ with $Z_F \in \real^{n \times k_F}$
and $k_F \le n$ as input.
However, the convergence properties as well as the computational
efficiency suffer from a large rank~$k_F$.
In~\cite[p.~10]{penzl}, the property $k_F \ll n$ is assumed in
the ADI method.

Concerning both systems~(\ref{simo-system}) and~(\ref{system-stable}),
an iterative computation of the reachability Gramian can be easily devised
because $F = B B^\top$ and thus $k_F = n_{\rm in}$
due to our assumption of a low number of inputs.
We obtain a low-rank approximation $P \approx Z_P Z_P^\top$
solving~(\ref{lyap-linear-P}) with $Z_P \in \real^{n \times k_P}$.
A low-rank factorization of the reachability Gramian also allows for
a fast computation of the value~(\ref{pprime}).

However, the observability Gramian requires a factorization
$F \approx Z_F Z_F^\top$
in the case of the linear dynamical system~(\ref{simo-system}),
where the rank may be large (possibly close to~$n$).
In the case of the quad\-ratic-bi\-linear system~(\ref{system-stable}),
the input matrix for~(\ref{lyap-linear-Q}) becomes
\begin{equation} \label{lyap-F}
  \hat{F} = SPS + 4 M B B^\top M \approx (S Z_P) (S Z_P)^\top +
  4 (MB) (MB)^\top .
\end{equation}
Now we obtain directly an approximate factorization
of~(\ref{lyap-F}) by
\begin{equation} \label{factor-F}
  Z_{\hat{F}} \approx \left( \; (S Z_P) \; , \; (2MB) \; \right)
  \in \real^{n \times (k_P+n_{\rm in})} ,
\end{equation}
where the number of columns is low since $k_P,n_{\rm in} \ll n$.
The rank of the factor~(\ref{factor-F}) may be smaller than $k_P+n_{\rm in}$,
which allows for a simplification to a full-rank factor
$Z_{\hat{F}} \in \real^{n \times k_{\hat{F}}}$ for some $k_{\hat{F}} < k_P + n_{\rm in}$.
Furthermore, a reduced factor~$Z_{\hat{F}}$ with $k_{\hat{F}}$ columns
can be obtained by using just the first $k_P' < k_P$ columns
of the factor~$Z_P$.
An iterative scheme solving~(\ref{lyap-linear-Q}) yields the factorization
$Q \approx Z_Q Z_Q^\top$ with $Z_Q \in \real^{n \times k_Q}$
for the observability Gramian.
For a general Lyapunov equation~(\ref{lyapunov-general}) with
$F = Z_F Z_F^\top$ and $Z_F \in \real^{n \times k_F}$,
$j$ iterations of the ADI method generate a factor with $j k_F$ columns,
see~\cite{penzl}.


The balanced truncation approach using approximate low-rank factors
represents a well-known strategy, see~\cite{gugercin-li}.
The detailed formulas can be found in~\cite{wolf}, for example.
We reduce to a dimension $r < n$, assuming the condition~(\ref{sv-large}) is satisfied.
Thus approximate factors $Z_P, Z_Q$ with ranks $k_P,k_Q \ge r-1$ are required
for the linear Lyapunov equations (\ref{lyap-linear-P}),(\ref{lyap-linear-Q}).
Symmetric decompositions $\tilde{P} \approx \tilde{Z}_{P} \tilde{Z}_{P}^\top$
and $\tilde{Q} \approx \tilde{Z}_{Q} \tilde{Z}_{Q}^\top$
for the matrices from (\ref{ansatzP}),(\ref{ansatzQ})
read as
$$ \tilde{Z}_{P} =
\begin{pmatrix}
  Z_P & 0 \\ 0 & \sqrt{\frac{p''}{2\varepsilon}} \\ 
\end{pmatrix}
\qquad \mbox{and} \qquad
 \tilde{Z}_{Q} = \frac{1}{\sqrt{2\varepsilon}}
\begin{pmatrix}
  Z_Q & 0 \\ 0 & 1 \\ 
\end{pmatrix} . $$
Now we compute an SVD of the small matrix
$\tilde{Z}_Q^\top \tilde{Z}_P \in \real^{(k_Q+1) \times (k_P+1)}$.
A partition~(\ref{svd-partition}) of this SVD is used again
assuming an ascending order of the singular values.
We suppose that the condition~(\ref{sv-large}) is also satisfied in
this approximation.
The projection matrices for the reduction~(\ref{reduced-matrices})
are
$$ \tilde{T}_{\rm l} =
\tilde{Z}_Q \tilde{U}_2 \tilde{\Sigma}_2^{-\frac{1}{2}}
\qquad \mbox{and} \qquad
\tilde{T}_{\rm r} =
\tilde{Z}_P \tilde{V}_2 \tilde{\Sigma}_2^{-\frac{1}{2}} $$
with $\tilde{U}_2 \in \real^{(k_Q+1) \times r}$
and $\tilde{V}_2 \in \real^{(k_P+1) \times r}$.
We require just the $r$~largest singular values and their singular vectors
for the computation of the ROM.

  \subsection{Algorithm overview}\label{ssec:algorithm}
  We summarize here the main steps of our method. We recall that we have rewritten the original MISO system \eqref{dynamicalsystem} into an equivalent quadratic-bilinear system with a single output \eqref{system-stable}. The following steps provide an algorithm for the construction of an ROM in form of a small quadratic-bilinear system:
  \begin{enumerate}
  \item Solve the linear Lyapunov equation \eqref{lyap-linear-P} for $P$, and subsequently the linear Lyapunov equation \eqref{lyap-linear-Q} for $Q$, where $S$ is defined in \eqref{symmetricmatrix}. This is the most costly portion of the MOR procedure, and in our numerical experiments we will solve these equations using two approaches: (i) direct linear algebraic methods, and (ii) approximate iterative methods, namely ADI iteration.
  \item Compute factorizations $P = L_P L_P^\top$ and $Q = L_Q L_Q^\top$,
    if not already obtained within the solution procedure of step~1.
    \item Choose a small stabilization parameter $\varepsilon > 0$ and compute $p''$ using \eqref{pprime}.
    \item Compute the SVD \eqref{svd-small} of $L_Q^\top L_P$.
    \item Assemble the matrices from \eqref{eq:SVD-tilde} used to form the SVD in \eqref{svd}.
    \item Given some rank $r \ll n$, perform the partition \eqref{svd-partition}, and subsequently form $\tilde{T}_{\rm l}$ and $\tilde{T}_{\rm r}$ in \eqref{eq:Tlr-truncation}.
    \item Construct the size-$r$ reduced system \eqref{system-reduced} via the matrices in \eqref{reduced-matrices}.
  \end{enumerate}
  In our overview above, we omit the technical details of many straightforward algebraic manipulations, which can be employed to substantially reduce the cost of direct computation in some of the steps. For example, one needs not explicitly assemble the full matrices in \eqref{eq:SVD-tilde} and just the dominant part of the SVD \eqref{svd-small} is required.


\section{Numerical results}
\label{sec:example}
We apply the reduction approaches from the previous sections 
to three test examples.
In each case,
three types of MOR using balanced truncation are examined:
\begin{itemize}
\item[i)] the linear dynamical system~(\ref{simo-system})
  with multiple outputs by direct algorithms of linear algebra,
\item[ii)] the quadratic-bilinear system~(\ref{system-stable})
  with single output by direct algorithms of linear algebra,
\item[iii)] the quadratic-bilinear system~(\ref{system-stable})
  with single output using ADI iteration.
\end{itemize}
The numerical computations were performed by the software package 
MATLAB \cite{matlab2017} (version R2016a),
where the machine precision is around $\varepsilon_0 = 2 \cdot 10^{-16}$.
We used the ADI algorithm from the Matrix Equation Sparse Solver (M.E.S.S.)
toolbox in MATLAB, see \cite{mmess}.
We note that alternative iterative approaches can effectively solve
the Lyapunov equations encountered in this paper (for example,
low-rank rational Krylov methods~\cite{druskin,kolesnikov}).
We focus on ADI methods in this paper for simplicity,
but acknowledge that alternative and perhaps better choices exist that
could improve the reported performance of our approach in this section.
The CPU times were measured on an iMAC with 3.4 GHz Inter Core i7 processor
and operation system OS X El Capitan.

In each test example, we compute discrete approximations of the
maximum absolute error and the integral mean value of the relative error
on a time interval $[0,T]$,
i.e., 
\begin{equation} \label{def-errors}
  E_{\rm abs} = \max_{t \in [0,T]} \left| \bar{y}(t) - y(t) \right|
  \qquad \mbox{and} \qquad
  E_{\rm rel} = \frac{1}{T} \int_0^T
  \frac{\left| \bar{y}(t) - y(t) \right|}{\left| y(t) \right|} \; {\rm d}t
\end{equation}
with $y$ from a full-order model (FOM) and $\bar{y}$ from an ROM.
On the one hand, the absolute error measures the maximum pointwise
discrepancy between the FOM trajectory and the ROM trajectory.
On the other hand, the relative error measures the discrepancy between
these two averaged over the trajectory.
Since~$y$ can have values close to zero, we expect that the
relative error can be large compared to the absolute error.
Our numerical results will observe this.

\subsection{Positive definite output matrix}
\label{sec:example-definite}
We construct a linear dynamical system~(\ref{dynamicalsystem})
of dimension $n=5000$.
A matrix $A' \in \real^{n \times n}$ is arranged using pseudo
random numbers with respect to a standard Gaussian distribution.
Let $\gamma$ be the largest real part of the eigenvalues of~$A'$.
We define the dense matrix
$A = A' - \lceil \gamma \rceil I$,
which implies an asymptotically stable system.
Furthermore, a single input is introduced using
the vector $B = (1,1,\ldots,1)^\top$.
We use the identity matrix $M=I$ in the definition of
the quadratic output.
This matrix is obviously symmetric and positive definite.
Even though this choice is simple, the identity matrix cannot be
well-approximated by a low-rank matrix.

As time-dependent input, we supply a chirp signal
\begin{equation} \label{chirp-signal}
  u(t) = \sin ( k(t) t ) \qquad \mbox{with} \qquad k(t) = k_0 t
\end{equation}
and the constant $k_0 = 0.1$.
All initial values are zero.
The total time interval $[0,100]$ is considered in the
transient simulations.
We use an explicit embedded Runge-Kutta method of convergence order 4(5) 
for computing numerical solutions of our initial value problems
(\texttt{ode45} in MATLAB).
This procedure uses step size selection by 
a local error control with relative tolerance
$\varepsilon_{\rm rel} = 10^{-6}$ and absolute tolerance
$\varepsilon_{\rm abs} = 10^{-8}$ in all state variables.
Thus high accuracy requirements are imposed.
Figure~\ref{fig:outputs} (left) shows the quadratic output
resulting from the numerical solution of~(\ref{dynamicalsystem}).

\begin{figure}
\begin{center}
\includegraphics[width=6cm]{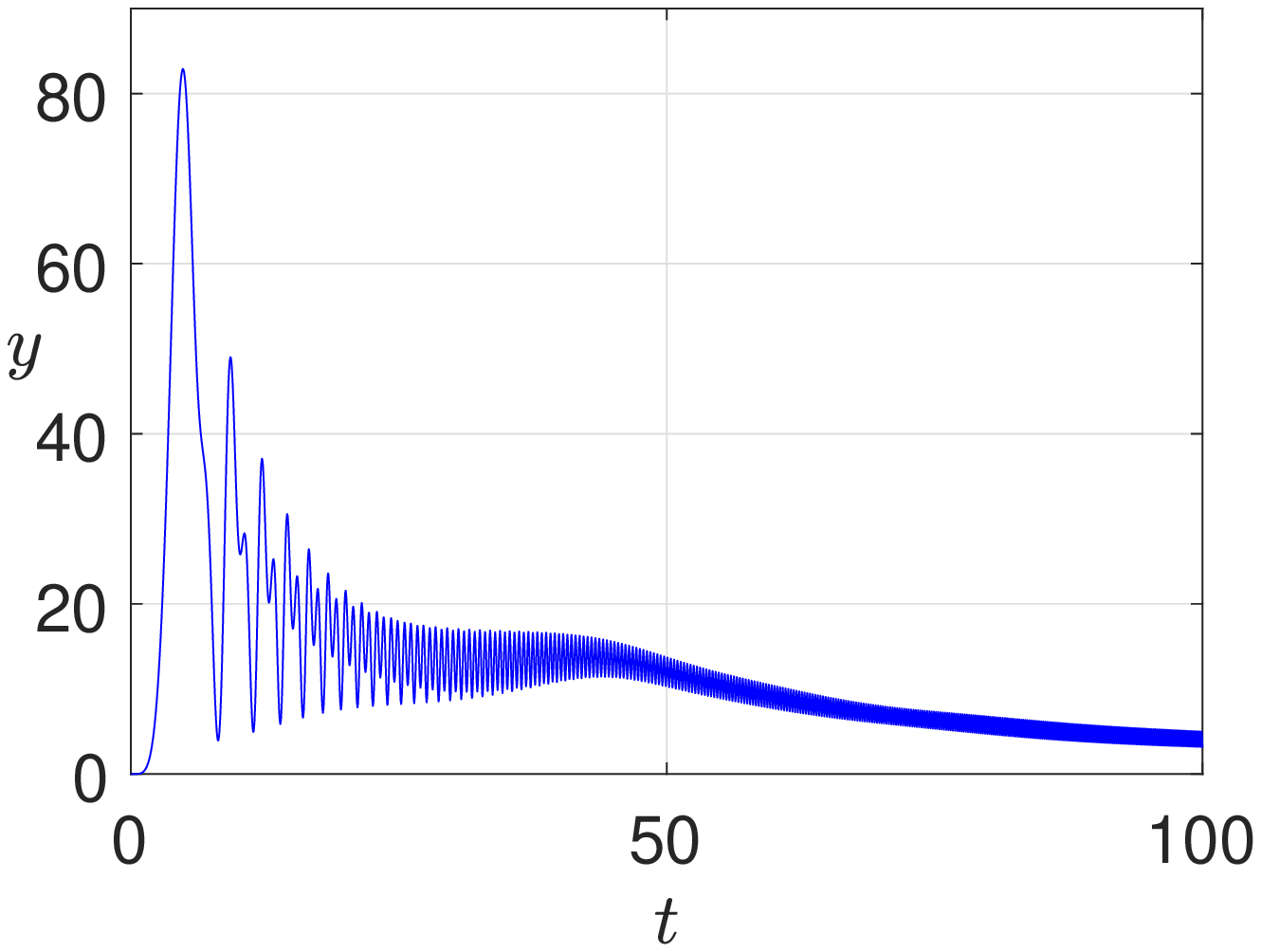}
\hspace{5mm}
\includegraphics[width=6cm]{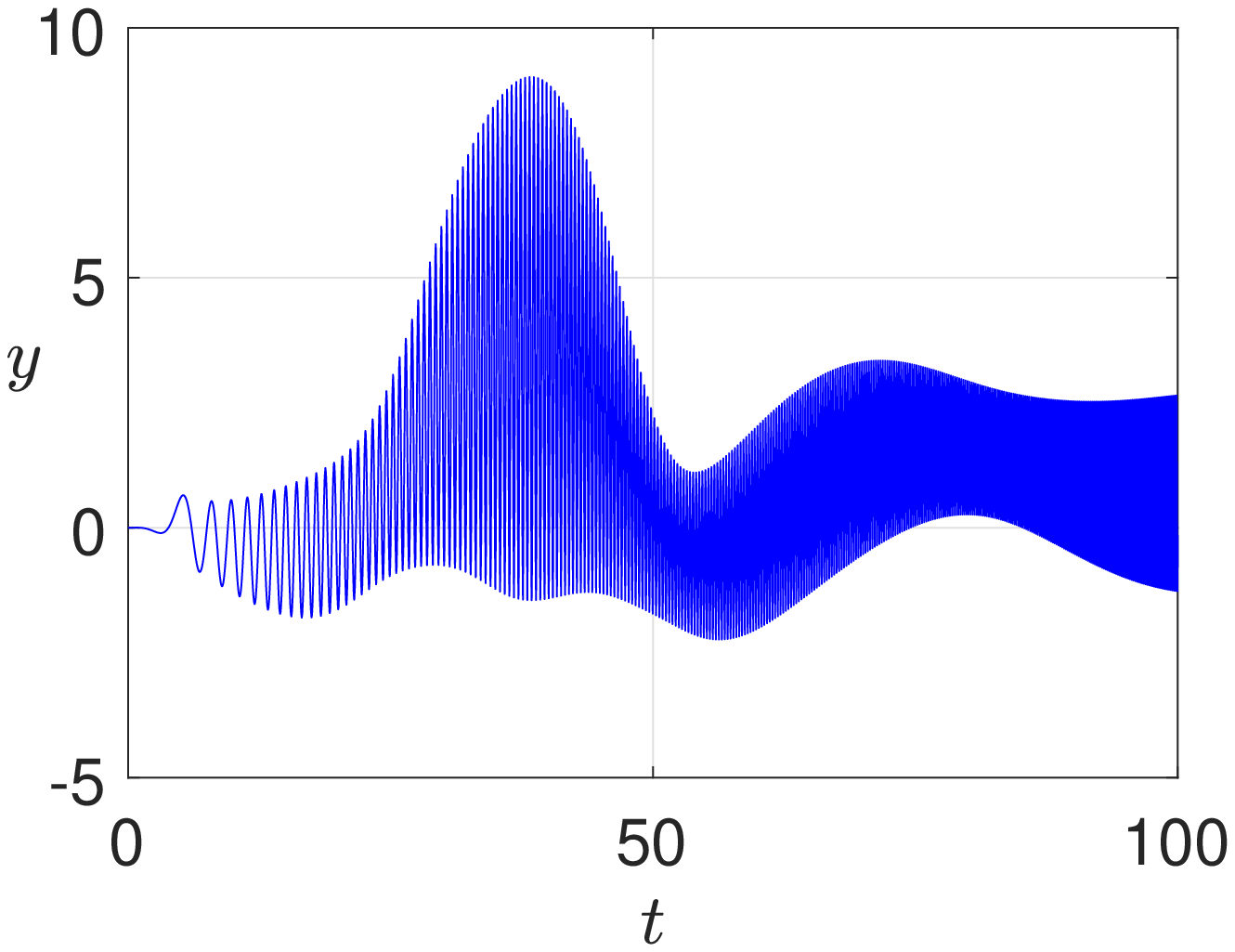}
\end{center}
\caption{Quadratic outputs of the linear dynamical systems for
  identity matrix (left) and indefinite matrix (right).}
\label{fig:outputs}
\end{figure}

On the one hand, we arrange the linear dynamical system~(\ref{simo-system}),
where it holds that $L^\top = I$.
Hence the number of outputs is equal to~$n$ in~(\ref{simo-system}).
The balanced truncation technique yields the Hankel singular values
in the linear case.
On the other hand,
we derive the quadratic-bilinear system~(\ref{system-stable})
including the stabilization parameter $\varepsilon = 10^{-8}$.
The balanced truncation scheme from Section~\ref{sec:decompositions}
produces other singular values.
The dominant singular values up to order~80 are illustrated
in descending magnitudes by Figure~\ref{fig:eye-sv} (left).
The largest singular value of the quadratic-bilinear
system~(\ref{system-stable}) has a special role, see~(\ref{sv-order}).
Thus we normalize the first singular value of~(\ref{simo-system}) and
the second singular value of~(\ref{system-stable}) to one.
Figure~\ref{fig:eye-sv} (right) shows the normalized singular values.
We observe the same rate of decay in both sets of singular values,
which indicates a similar potential for MOR by balanced truncation.

\begin{figure}
\begin{center}
\includegraphics[width=6cm]{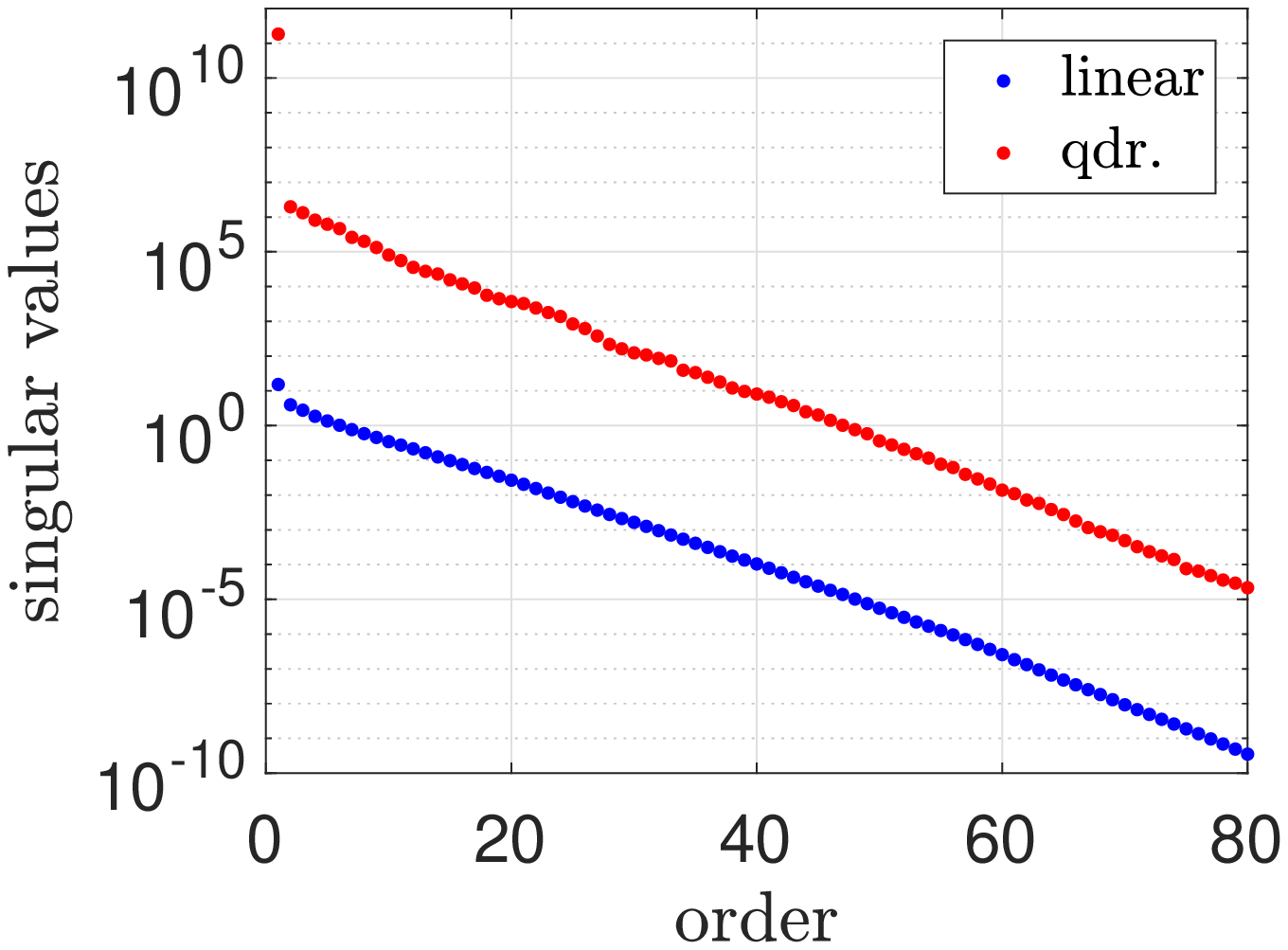}
\hspace{5mm}
\includegraphics[width=6cm]{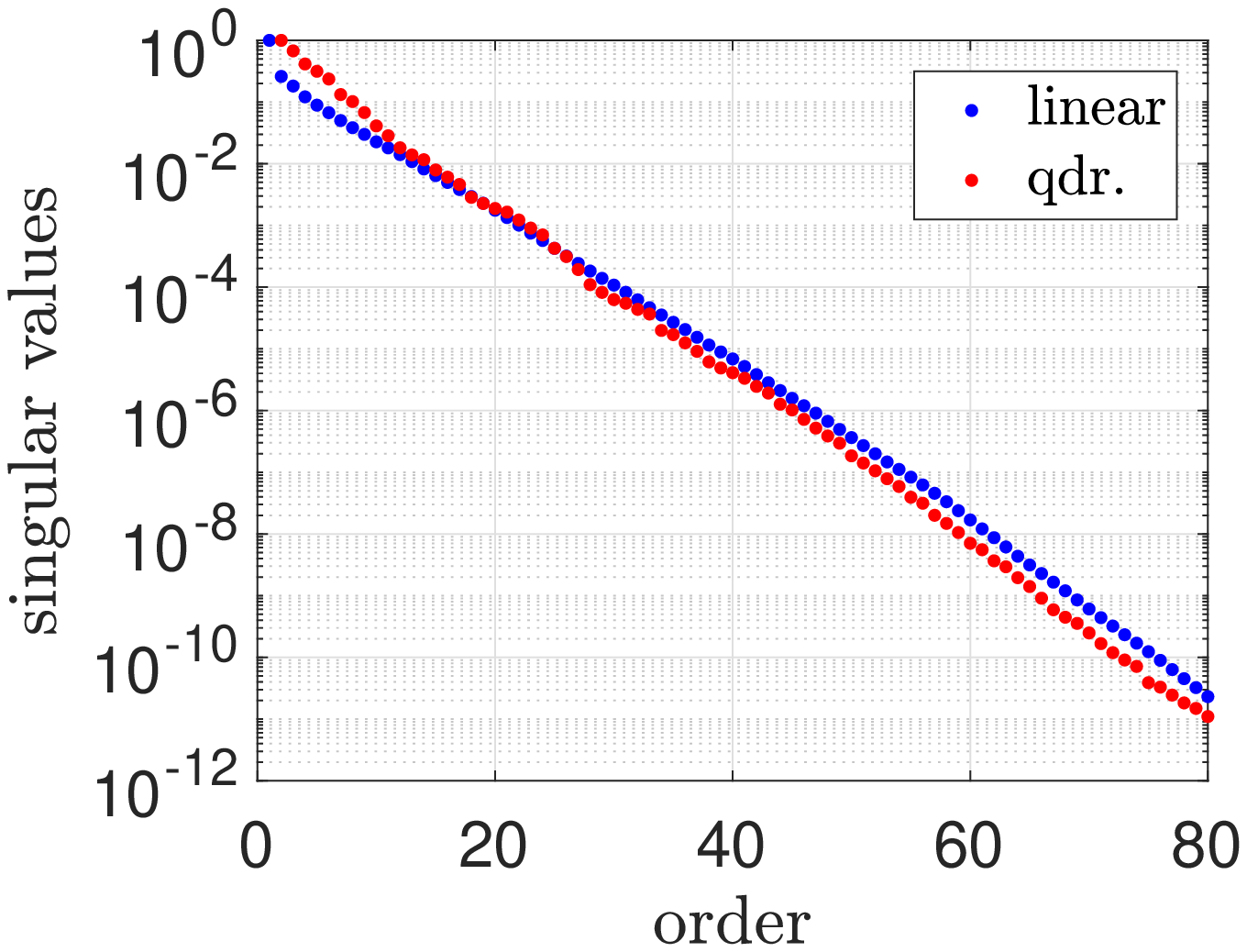}
\end{center}
\caption{Singular values (left) and their normalized values (right)
  for the two dynamical systems associated to identity output matrix.}
\label{fig:eye-sv}
\end{figure}

Since we apply direct linear algebra methods, each balanced truncation
approach yields square transformation matrices,
see~(\ref{square-transformation}) for the quadratic-bilinear case.
The projection matrices of the MOR result from the dominant
columns of these square matrices. 
Consequently, we obtain an ROM for an arbitrary dimension $r<n$.
We compute ROMs from the linear dynamical systems~(\ref{simo-system})
and from the quadratic-bilinear system~(\ref{system-stable})
for $r=5,6,\ldots,80$.
In the quadratic-bilinear case, we choose $\varepsilon = 0$
only in the reduction of the matrix~(\ref{barA}),
because numerical results show that the errors become smaller
as $\varepsilon \rightarrow 0$ in this matrix.

Furthermore, we solve the linear Lyapunov equations~(\ref{lyap-linear-P})
and~(\ref{lyap-linear-Q}) iteratively by an ADI method
for each~$r$ separately as described in Section~\ref{sec:lowrank}.
On the one hand,
the low-rank factor~$Z_P$ of the reachability Gramian is computed 
with rank $k_P = r + 10$ by $j=k_P$ iteration steps.
On the other hand,
only the first $r$~columns of~$Z_P$ are inserted in the
Lyapunov equation~(\ref{lyap-linear-Q}) and $j=10$ iteration steps
are performed.
It follows that the low-rank factor~$Z_Q$ of the observability Gramian
has rank $k_Q = j(r+1)$ due to~(\ref{factor-F}).
Each pair of iterative solutions implies an associated ROM.

The transient simulation of the FOM~(\ref{dynamicalsystem})
yields approximations at many time points with variable step sizes.
We integrate the ROMs by the same Runge-Kutta method with
local error control including the tolerances from above.
The integrator produces approximations of identical convergence order
at the predetermined time points.
We obtain discrete approximations of the errors~(\ref{def-errors})
by the differences in each time point,
which are depicted in Figure~\ref{fig:eye-error}.
Both the maximum absolute error and the mean relative error decrease
exponentially for increasing dimensions of the ROMs.
The errors start to stagnate at higher dimensions,
since the errors of the time integration become dominant.
Furthermore, tiny values of the exact solution appear close to the
initial time, which locally causes large relative errors.
The errors are lower for the linear dynamical system~(\ref{simo-system})
in comparison to the quadratic-bilinear system~(\ref{system-stable}),
although the associated singular values exhibit a similar rate of decay.
We suspect that the balanced truncation strategy works in general better for linear dynamical systems.
This suspicion can be corroborated by other facts: For example, ROMs for linear dynamical systems have error bounds that depend on the singular values, but error bounds depending on the singular values of a quadratic-bilinear case are not known.
The iterative solution of the Lyapunov equations associated to
the quadratic-bilinear system results in larger errors than the
direct solution, since several approximations are made in this procedure.

\begin{figure}
\begin{center}
\includegraphics[width=6cm]{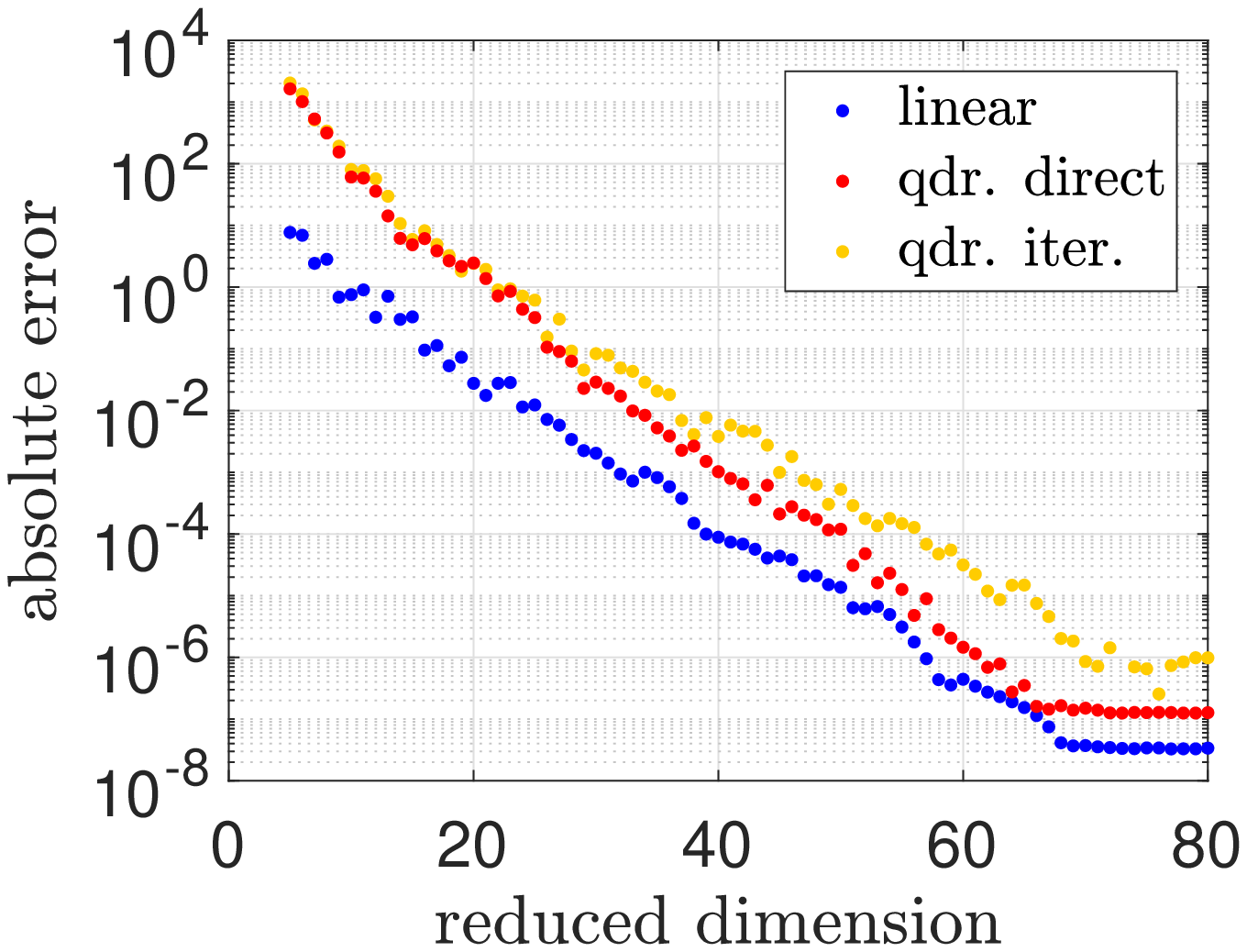}
\hspace{5mm}
\includegraphics[width=6cm]{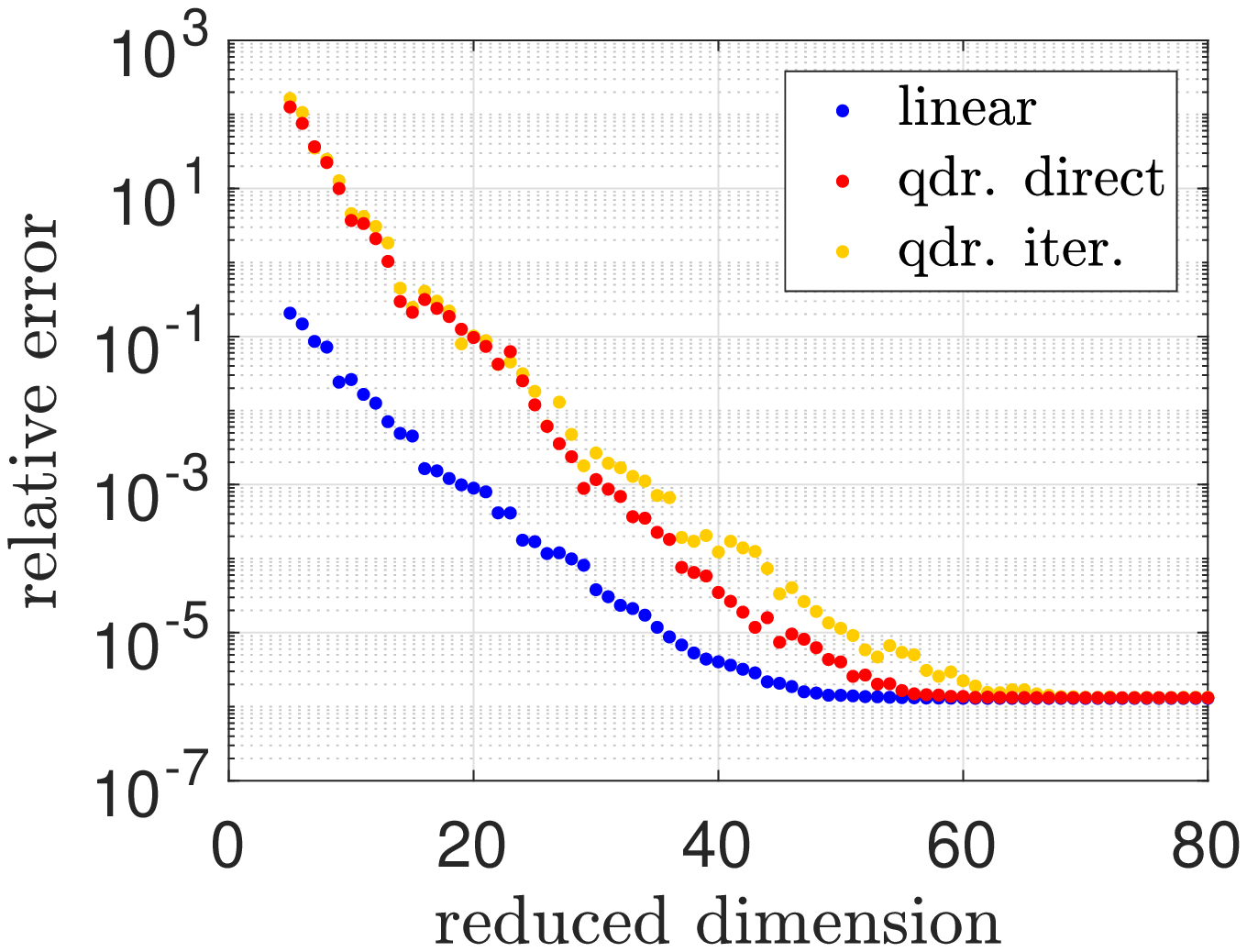}
\end{center}
\caption{Maximum absolute errors (left) and mean relative errors (right)
  for ROMs of different dimensions in the case of identity output matrix.}
\label{fig:eye-error}
\end{figure}
 
Now we analyze the CPU times of the three MOR techniques.
The CPU time for the time integration of the FOM~(\ref{dynamicalsystem})
is 8602.9 seconds.
Figure~\ref{fig:eye-cputimes1} (left) illustrates the effort for the
calculation of the projection matrices in the balanced truncation,
which includes the solution of Lyapunov equations and thus represents
the main part of the computation work.
In the direct approaches, the complete transformation matrices have to
be computed, which is indicated by constant CPU times.
In the iterative approach, the effort grows just slowly with
increasing reduced dimension.
Figure~\ref{fig:eye-cputimes2} depicts the CPU time for the
computation of the reduced matrices (left) \textit{and} the time integration
of the ROMs including the calculation of
the quantity of interest (right).
We observe that the computation work for the matrices is negligible.
The total speed up is shown in Figure~\ref{fig:eye-cputimes1} (right),
where both the construction and the transient simulation of the ROMs
is compared to the solution of the FOM.
The speedup is nearly constant for different reduced dimensions in the
direct approaches because the balanced truncation part dominates.
The iterative method exhibits a significantly higher speedup for
small dimensions,
whereas the speedup decreases for larger dimensions.

\begin{figure}
\begin{center}
\includegraphics[width=6cm]{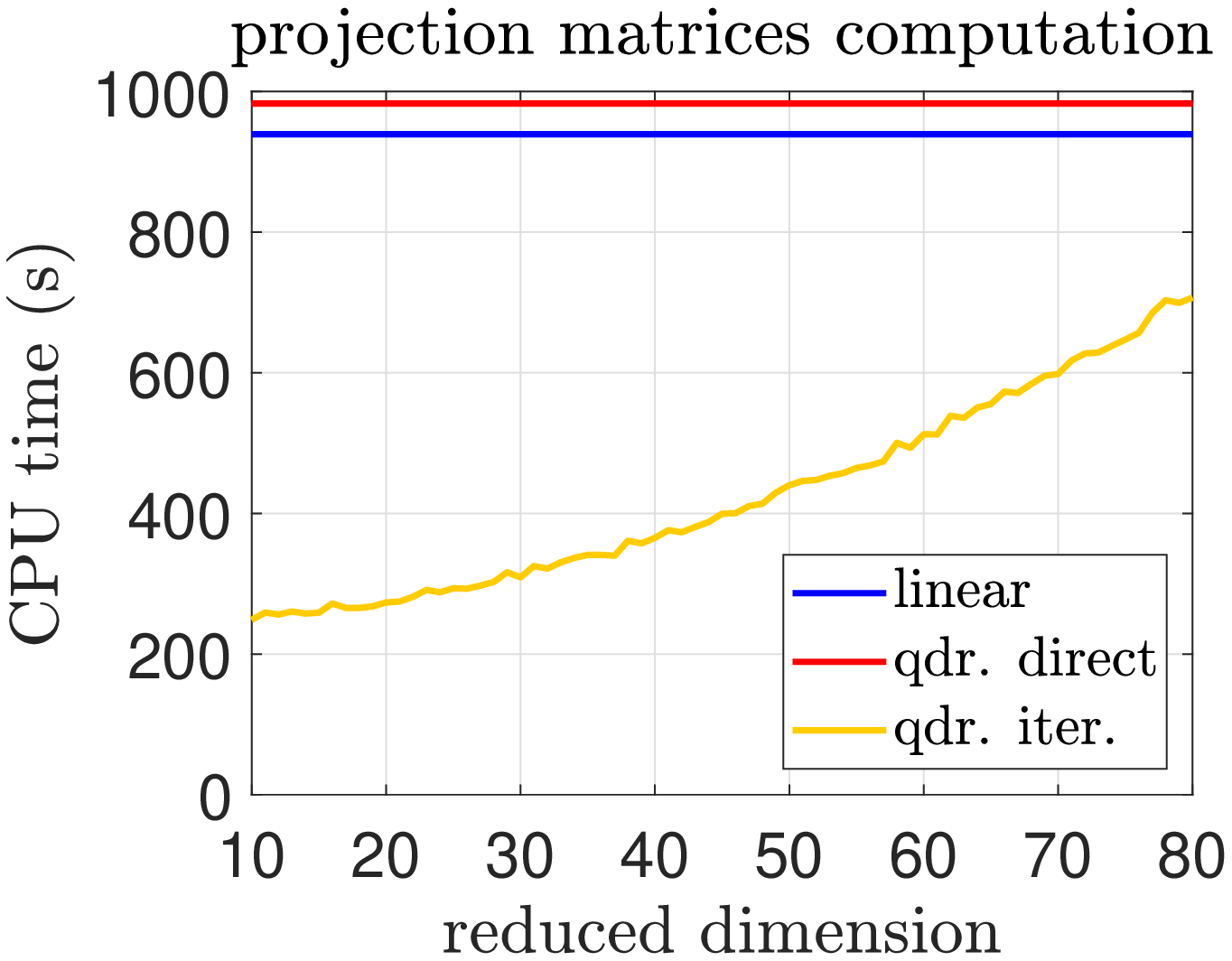}
\hspace{5mm}
\includegraphics[width=6cm]{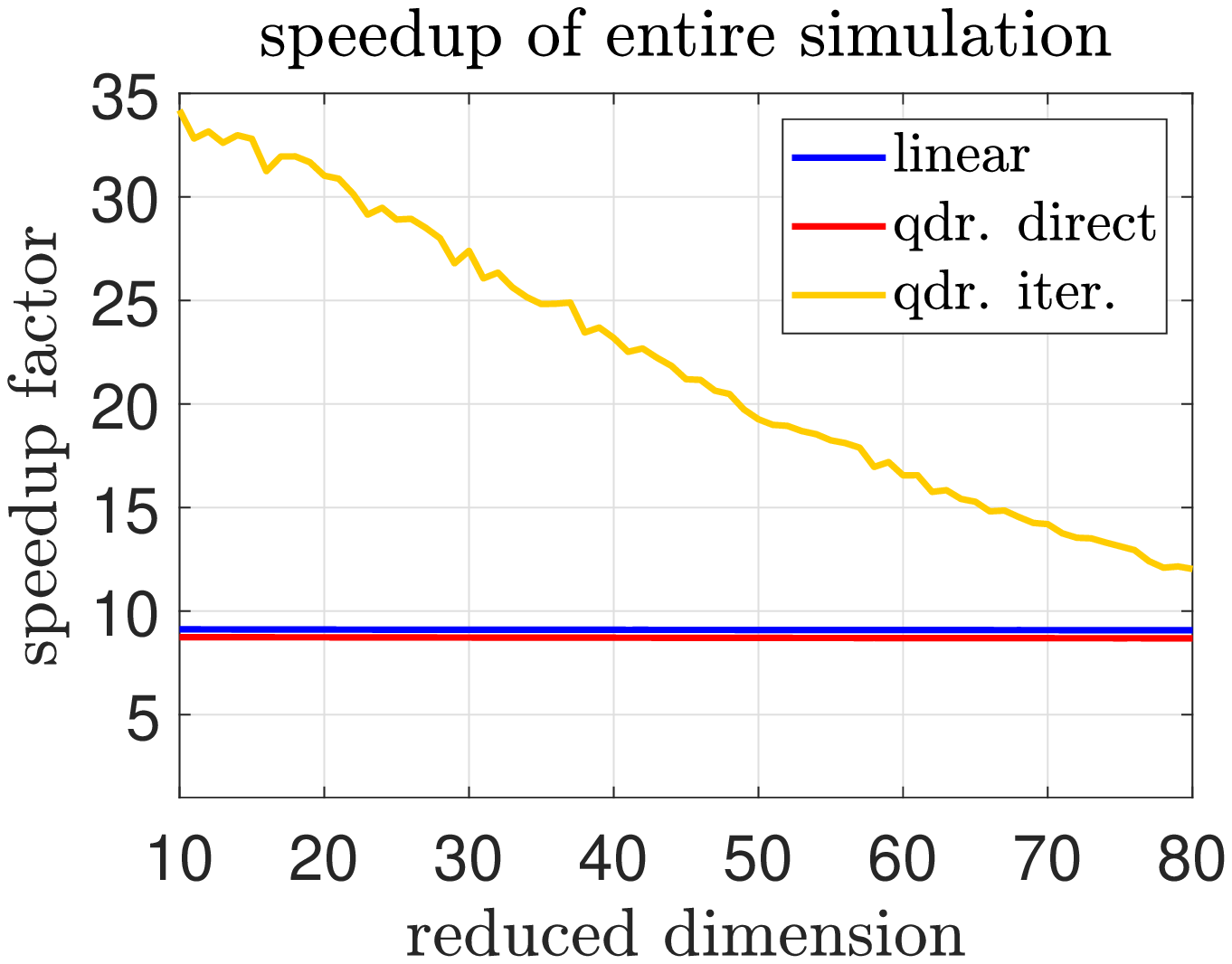}
\end{center}
\caption{CPU times for computation of projection matrices (left)
  and total speed ups (right) in the case of identity
  output matrix.}
\label{fig:eye-cputimes1}
\end{figure}

\begin{figure}
\begin{center}
\includegraphics[width=6cm]{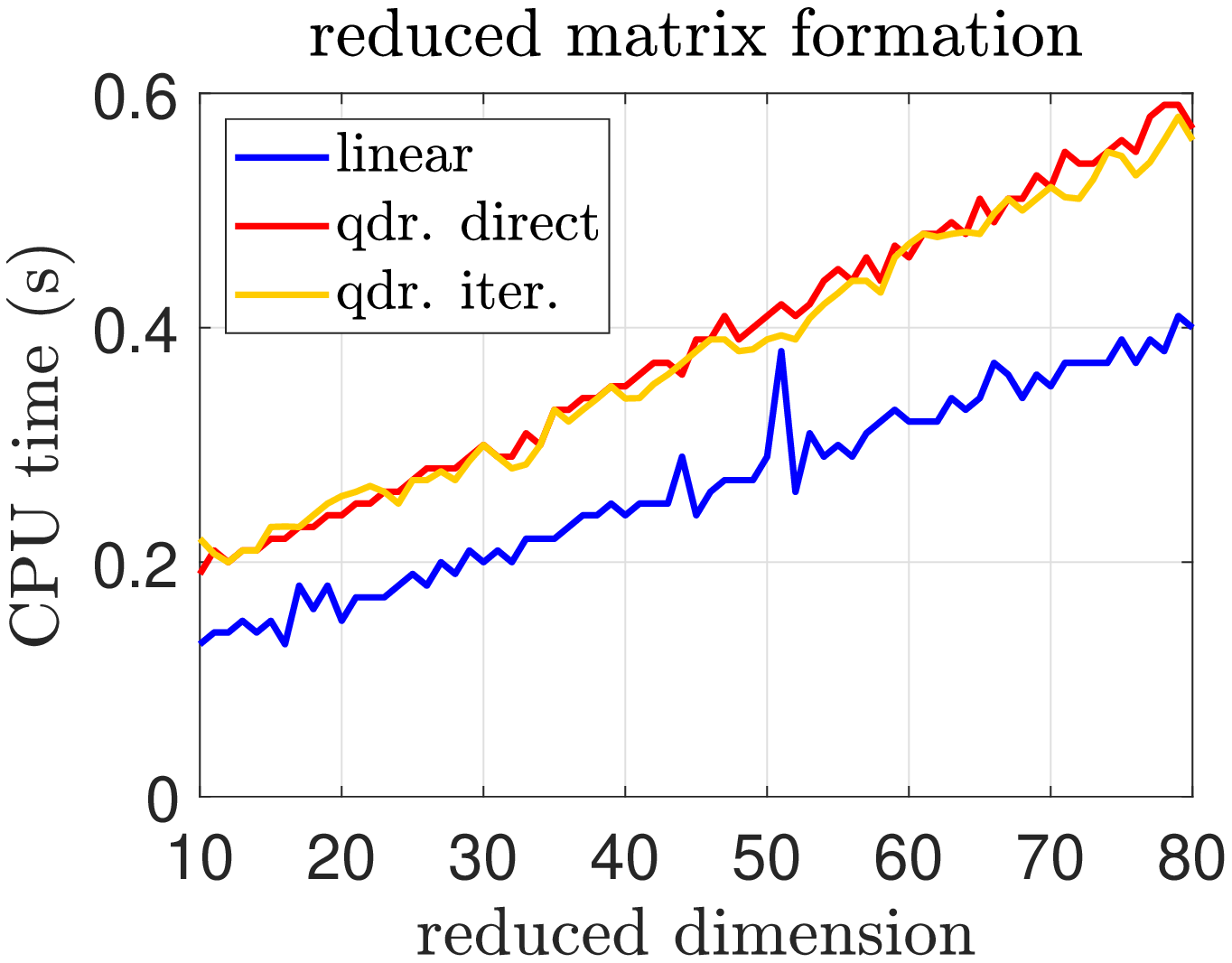}
\hspace{5mm}
\includegraphics[width=6cm]{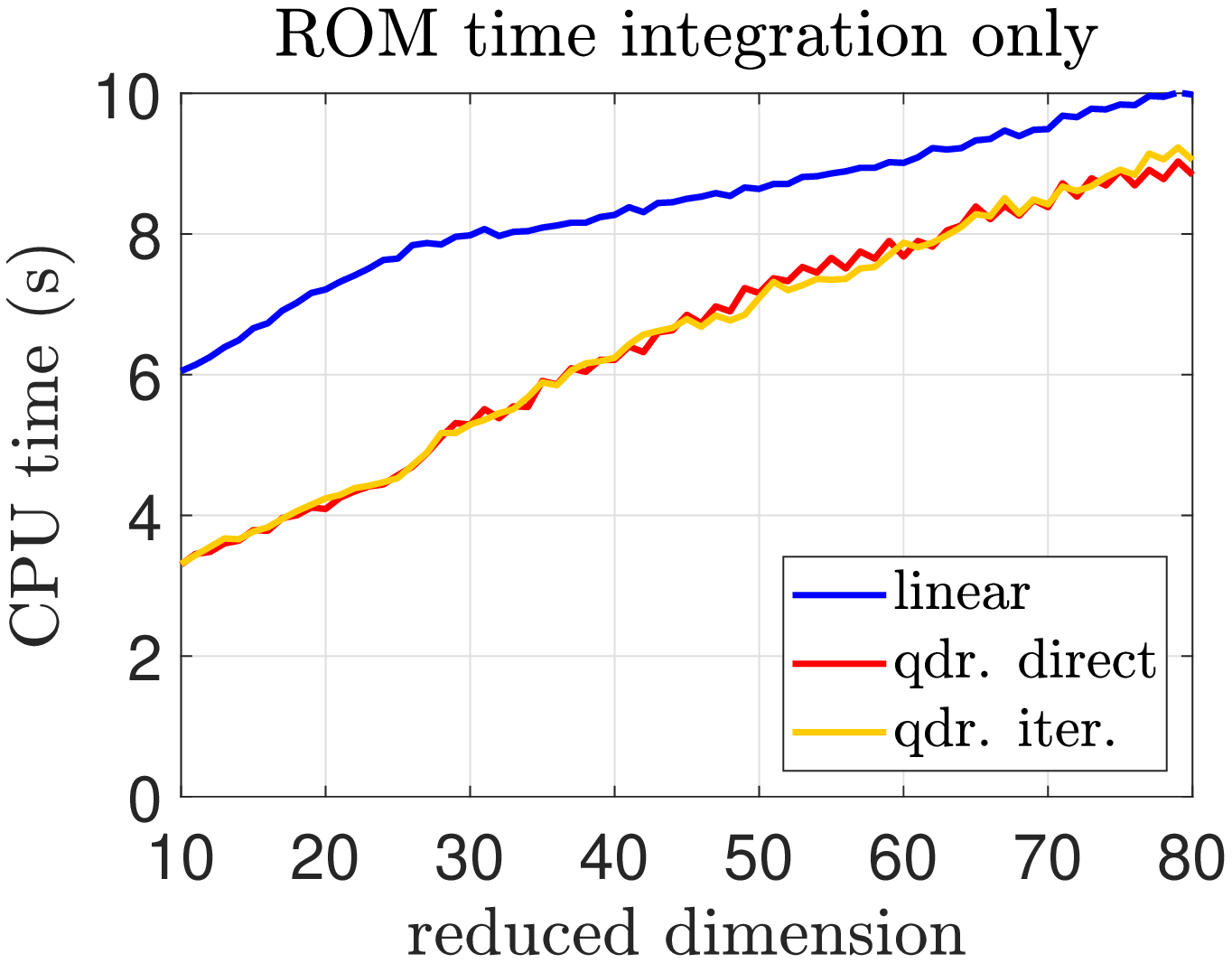}
\end{center}
\caption{CPU times for computation of reduced matrices (left)
  and transient simulation of ROMs (right) in the case of identity
  output matrix.}
\label{fig:eye-cputimes2}
\end{figure}

We conclude that the quadratic-bilinear approach using iterative solvers can achieve error comparable to the linear approach with substantially increased computational efficiency. For a fixed reduced dimension, choosing the quadratic approach decreases accuracy when compared to the linear approach; this accuracy drop is slightly exacerbated by use of an iterative solver for computing the Gramian matrices, cf. Figure~\ref{fig:eye-error}. However, the quadratic bilinear approach using an iterative solver is significantly faster than the linear approach. It is so efficient that one can compute a quadratic-bilinear reduced order model of significantly increased rank (and hence accuracy) for a fixed computational budget. For example, to achieve a relative error of approximately $10^{-5}$, we require a reduced dimension of approximately $r \approx 50$ for the quadratic-bilinear case, see Figure~\ref{fig:eye-error} (right). However, even at reduced rank 50, the quadratic-bilinear iterative approach is still approximately twice as fast as the linear or quadratic-bilinear direct approaches, see Figure~\ref{fig:eye-cputimes1} (right).

\begin{table} 
  \caption{Maximum difference between outputs from ROMs of dimension $r=20$
    for different stabilization parameters~$\varepsilon$ to
    reference value $\varepsilon = 10^{-8}$ 
    (differences are rounded to one digit).\label{tab:par}}  
  \begin{center}
    \begin{tabular}{lccccccc}
      value~$\varepsilon$ & $10^{-1}$ & $10^{-2}$ & $10^{-3}$ & $10^{-4}$ &
      $10^{-5}$ & $10^{-6}$ & $10^{-7}$ \\ \hline
      difference & $4 \cdot 10^{-8}$ &
      $3 \cdot 10^{-8}$ & $8 \cdot 10^{-9}$ & $5 \cdot 10^{-9}$ &
      $2 \cdot 10^{-8}$ & $2 \cdot 10^{-8}$  & $2 \cdot 10^{-8}$ \\
    \end{tabular}
  \end{center}
\end{table}

Finally, we investigate the choice of different stabilization
parameters~$\varepsilon$.
The focus is on ROMs of dimension $r=20$, where the transient outputs
are computed by the time integration described above.
The value $\varepsilon=0$ is used in the reduced matrices~(\ref{barA}) again.
We compute the maximum difference in time between the outputs for
different~$\varepsilon$ to the reference value $\varepsilon=10^{-8}$.
Table~\ref{tab:par} depicts the numerical results.
The differences are tiny and exhibit the same order or magnitude for
all~$\varepsilon$.
This property confirms the theoretical results in
Section~\ref{sec:decompositions}, which imply that the ROMs are independent
of~$\varepsilon$ except for the scalar entry in~(\ref{barA}).

\subsection{Indefinite output matrix}
We arrange a linear dynamical system~(\ref{dynamicalsystem})
of dimension $n=5000$ with a system matrix~$A$ and a vector~$B$
as in Section~\ref{sec:example-definite}.
(But using a different realization of the pseudo random numbers.)
We fill a matrix $M' \in \real^{n \times n}$ by pseudo random numbers
associated to a uniform distribution in $[-1,1]$.
Now the output matrix~$M := \frac{1}{2} (M' + M'^\top)$ is
dense, full-rank and symmetric.
The matrix is indefinite, and for our realization it has exactly $\frac{n}{2}$
positive eigenvalues and $\frac{n}{2}$ negative eigenvalues.
Thus the construction of the output matrix in the
linear dynamical system~(\ref{simo-system}) requires the
complete eigen-decomposition of~$M$.
We apply the chirp signal~(\ref{chirp-signal}) with $k_0=0.1$
as input again.
Figure~\ref{fig:outputs} (right) illustrates the quadratic output
of the system~(\ref{dynamicalsystem}),
which exhibits both positive and negative values.

We use balanced truncation by direct linear algebra methods
for the linear dynamical system~(\ref{simo-system}) with $n$ outputs and
the quadratic-bilinear system~(\ref{system-stable}) with single output,
where the stabilization parameter is $\varepsilon=10^{-8}$.
The resulting dominant singular values are depicted
in Figure~\ref{fig:ind-sv}.
Again the rate of decay is nearly identical in both cases.

\begin{figure}
\begin{center}
\includegraphics[width=6cm]{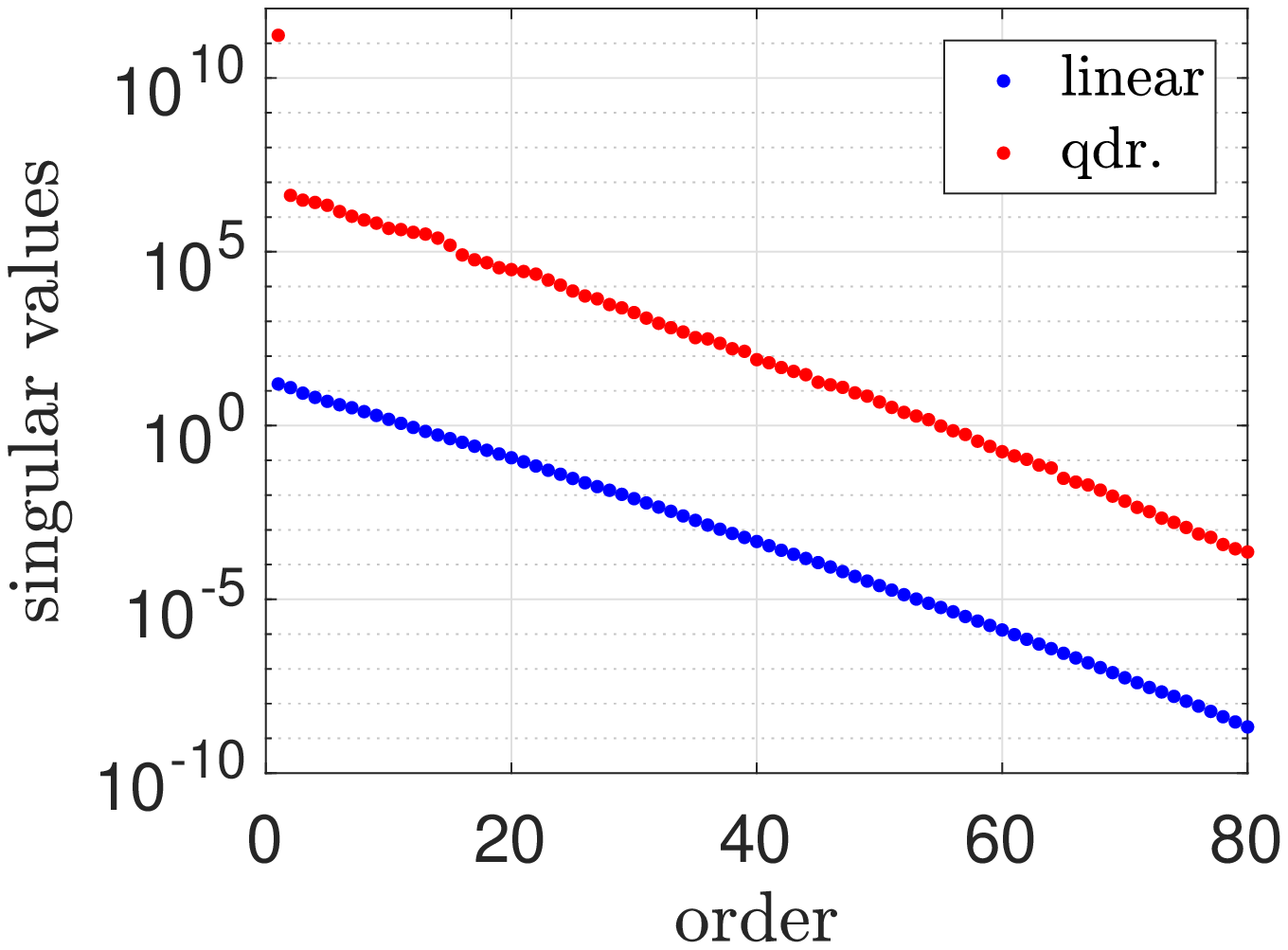}
\hspace{5mm}
\includegraphics[width=6cm]{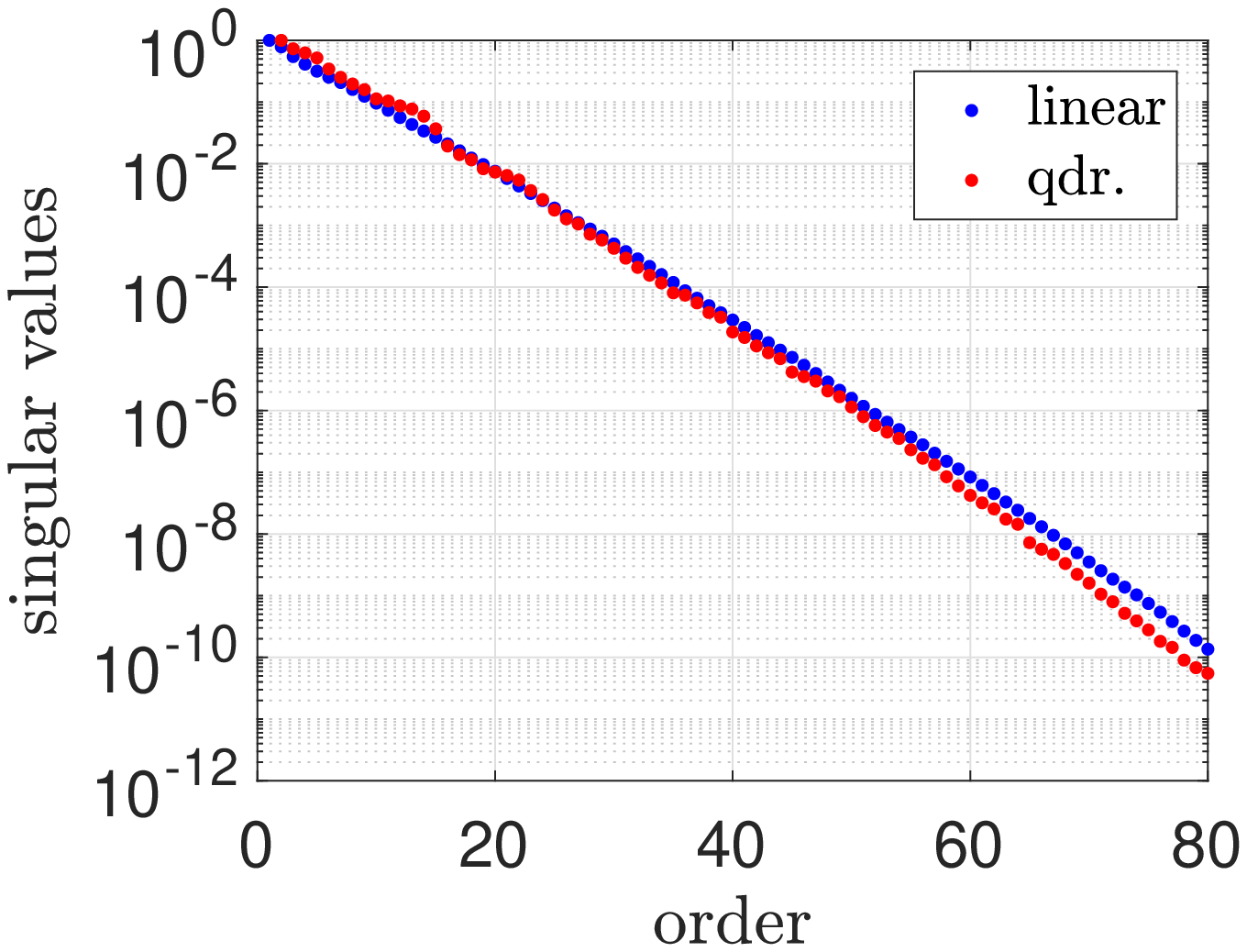}
\end{center}
\caption{Singular values (left) and their normalized values (right)
  for the two dynamical systems with respect to indefinite output matrix.}
\label{fig:ind-sv}
\end{figure}

The ROMs are computed as in Section~\ref{sec:example-definite}.
The same Runge-Kutta method with identical tolerances is used
for solving the initial value problems, where all initial values
are zero.
Figure~\ref{fig:ind-error} shows the resulting discrete approximations
of the error measures~(\ref{def-errors}).
The behavior of the errors is similar to the previous example in
Section~\ref{sec:example-definite}.
Concerning the relative errors, note that the exact quadratic output
features many zeros.

\begin{figure}
\begin{center}
\includegraphics[width=6cm]{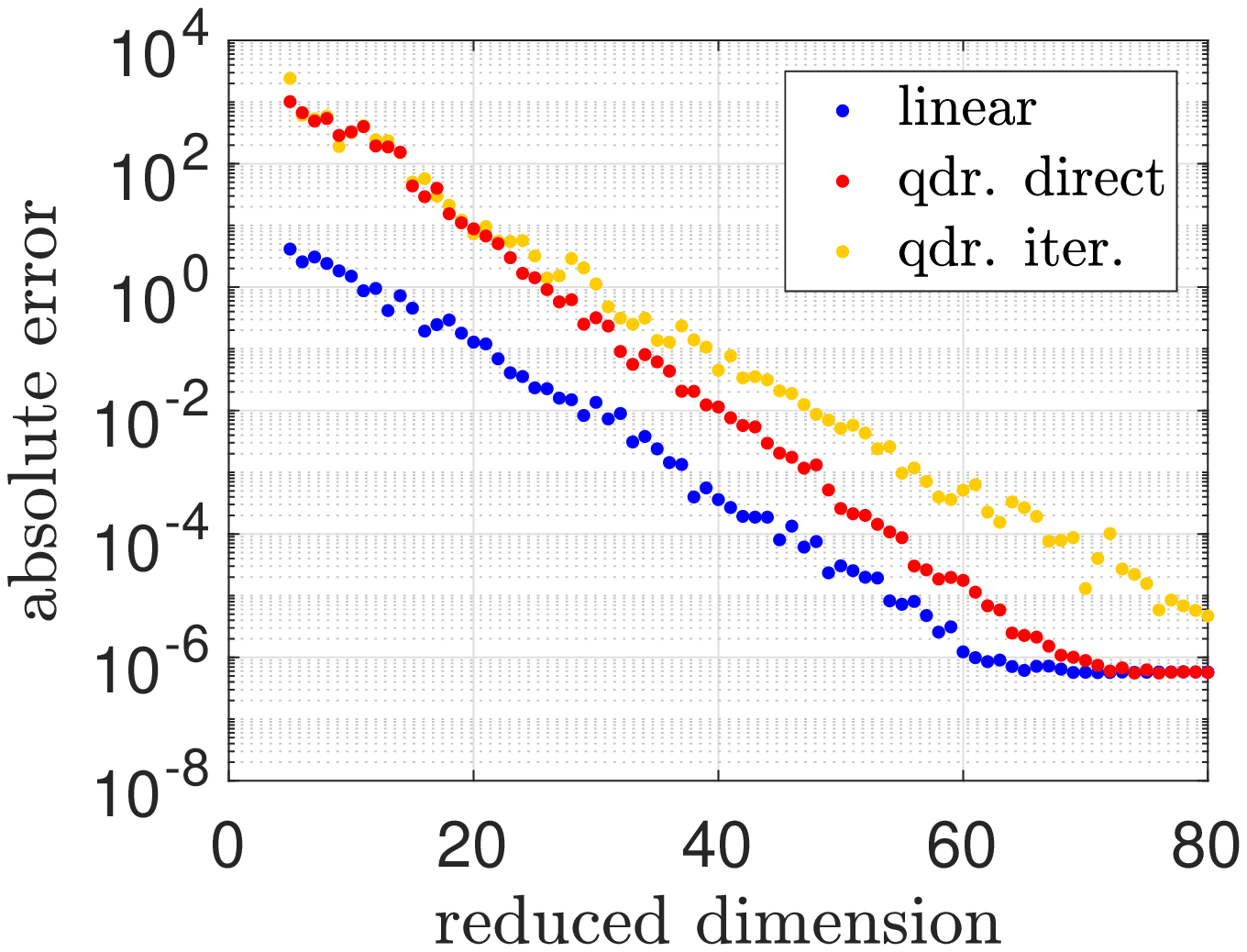}
\hspace{5mm}
\includegraphics[width=6cm]{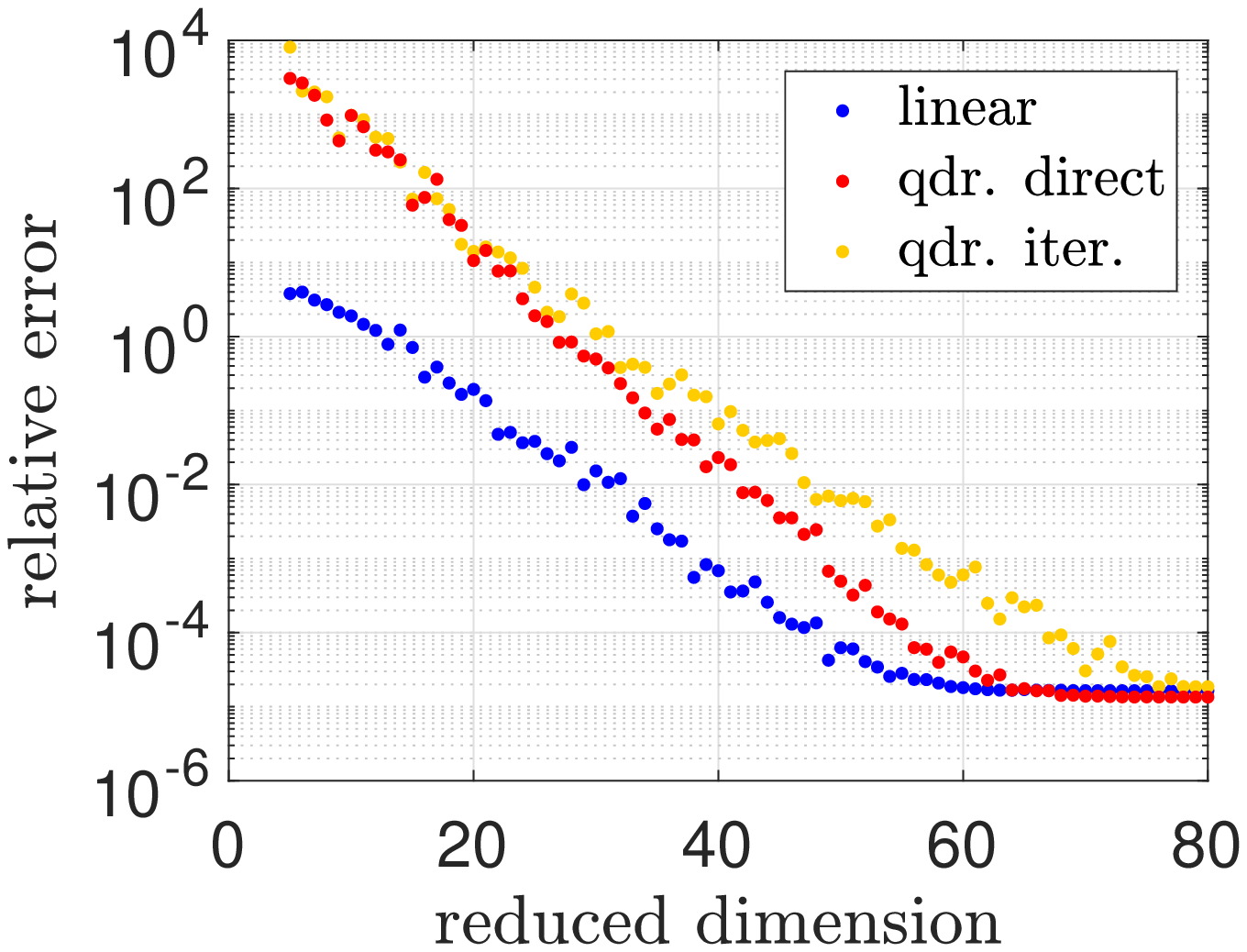}
\end{center}
\caption{Maximum absolute errors (left) and mean relative errors (right)
  for ROMs of different dimensions in the case of indefinite output matrix.}
\label{fig:ind-error}
\end{figure}


\subsection{Stochastic Galerkin method and variance}
We consider a mass-spring-damper configuration
from Lohmann and Eid~\cite{lohmann-eid}.
The associated linear dynamical system consists of
8~ordinary differential equations including 14~physical parameters
and is single-input-single-output (SISO).
In~\cite{pulch17}, this test example was investigated in the context of stochastic modeling, where
the physical parameters are replaced by independent random variables.
The state variables as well as the output are expressed
as a polynomial chaos expansion with $m=680$ basis polynomials, see~\cite{xiu-book}.

The stochastic Galerkin approach yields a larger linear dynamical system
(SIMO)
\begin{align} \label{galerkin}
  \begin{split}  
  \dot{x}(t) & = A x(t) + B u(t) \\[1ex] 
  w(t) & = C x(t), 
  \end{split}
\end{align} 
$x \in \real^n$, with dimension $n=5440$ and outputs $w = (w_1,\ldots,w_m)^\top$.
The first output~$w_1$ represents an approximation of the expected value
for the original single output.
The other outputs $w_2,\ldots,w_m$ produce an approximation of its
variance by
\begin{equation} \label{variance}
  {\rm Var}(t) = \sum_{i=2}^m w_i(t)^2 .
\end{equation}
The details of the above modeling can be found in~\cite{pulch17}.

As single input, we choose the harmonic oscillation
$u(t) = \sin ( \omega t )$
with frequency $\omega = 0.2$.
Initial value problems are solved with starting values zero
in the time interval $[0,T]$ with $T=2000$.
Since the stochastic Galerkin system~(\ref{galerkin}) is mildly stiff,
we apply the implicit trapezoidal rule.
Figure~\ref{fig:statistics} shows the approximations of the
expected value as well as the variance obtained from
the transient simulation.
Driven by the input signal,
the solutions become nearly periodic functions after an initial phase.

\begin{figure}
\begin{center}
\includegraphics[width=6cm]{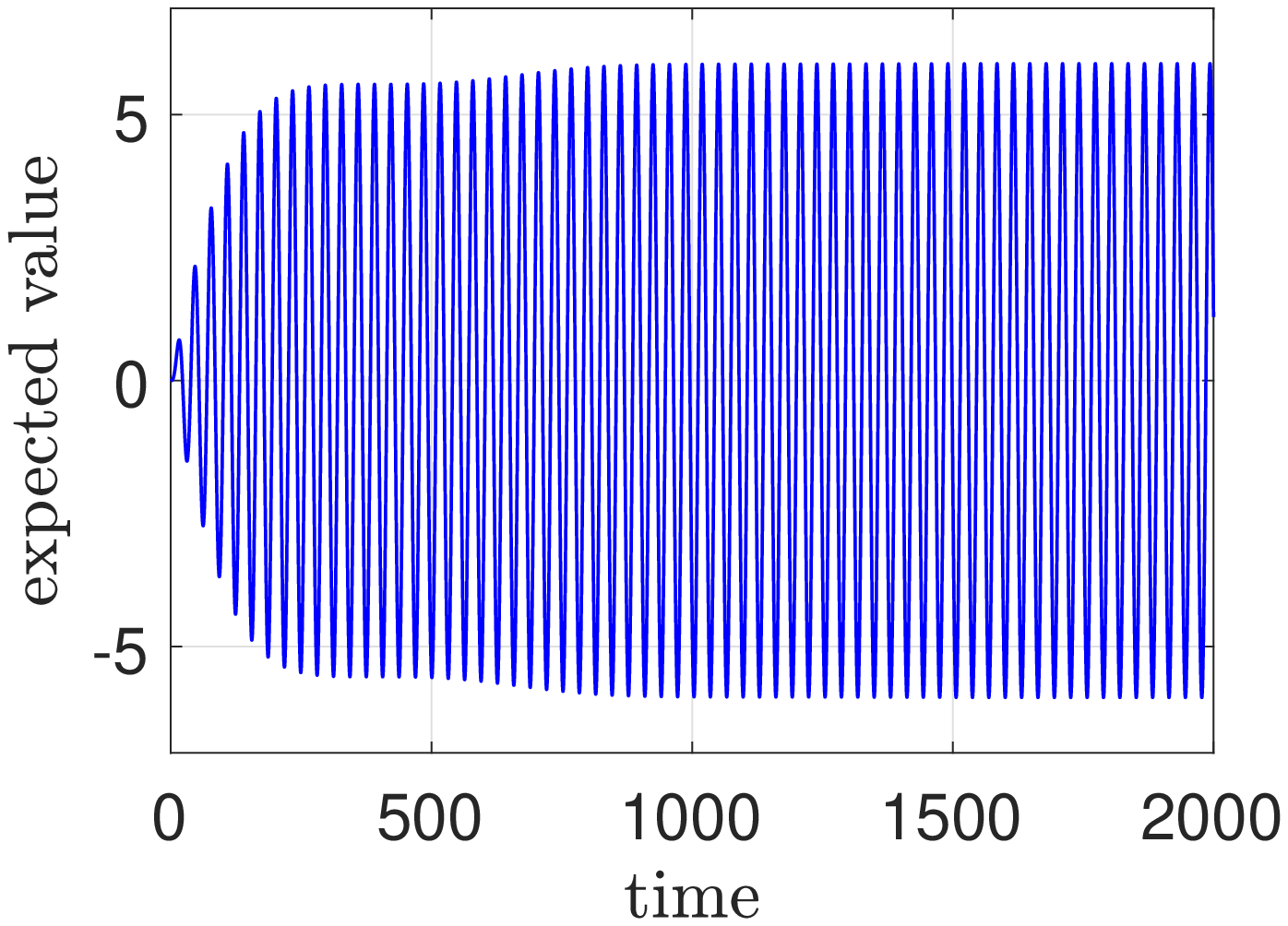}
\hspace{5mm}
\includegraphics[width=6cm]{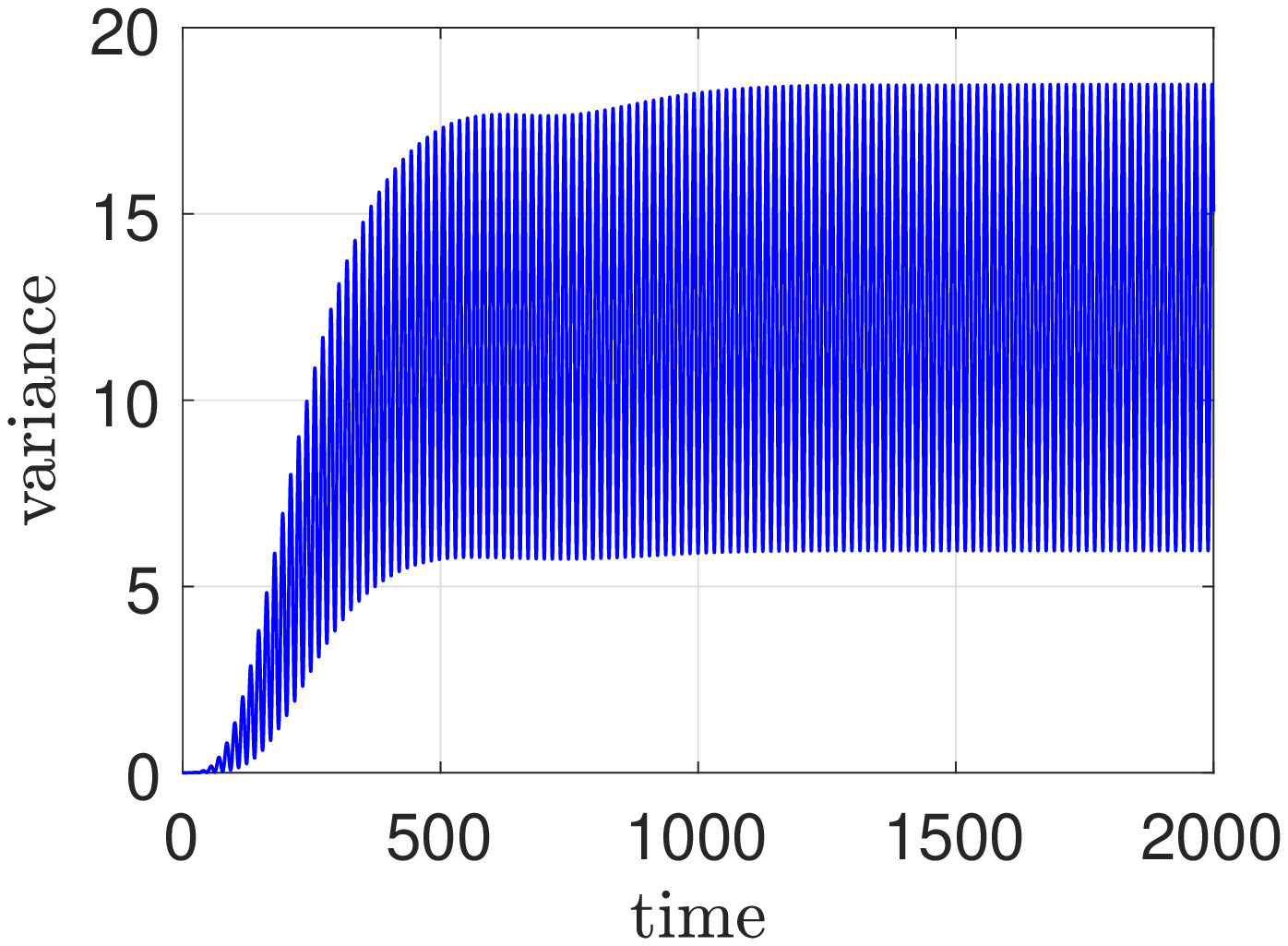}
\end{center}
\caption{Expected value (left) and variance (right) of random output
  in mass-spring-damper configuration.}
\label{fig:statistics}
\end{figure}
 
We construct a linear dynamical system~(\ref{dynamicalsystem})
from the stochastic Galerkin system~(\ref{galerkin}),
whose quadratic output is the variance~(\ref{variance}).
Define $L^\top \in \real^{(m-1) \times n}$ as the matrix~$C \in \real^{m \times n}$
in~(\ref{galerkin}) with its first row omitted.
It follows that $M = L L^\top$ in~(\ref{dynamicalsystem})
is symmetric and positive semi-definite of rank $m-1$.
Thus the equivalent system~(\ref{simo-system}) with $m-1$ linear outputs
defined by~$L^\top$ is already available in this application.

The associated quadratic-bilinear system~(\ref{qdr-bil-system}) is  
without bilinearity, because the property~(\ref{N-is-zero}) is satisfied.
We apply the stabilized system~(\ref{system-stable}) with
a parameter $\varepsilon = 10^{-8}$ again.

We examine the MOR by balanced truncation for this problem
comparing the reduction of the linear system~(\ref{simo-system})
and the quadratic system~(\ref{system-stable}).
In the linear system, projection matrices are obtained directly
by linear algebra algorithms.
In the quadratic system, both a direct method and
an iterative method using the ADI technique are employed.
We compute projection matrices for the reduced dimension $r_{\max} = 100$
in each approach.
Within the ADI iteration, an approximate factor for the
reachability Gramian is computed with rank $k_P = 200$.
Just the first $k_P' = 20$ columns are applied with $j = 20$ iterations
for the calculation of an approximate factor of the observability Gramian
with rank $k_Q = 400$.

The balanced truncation techniques yield singular values 
in each of the three reductions,
which are the Hankel singular values in the linear case.
Figure~\ref{fig:msd-sv} illustrates the dominant singular values
in descending order.
We recognize a faster decay of the singular values in
the quadratic system~(\ref{qdr-bil-system}).
However, the faster decrease of the singular values in the iterative method
represents an error by the approximation, because the direct approach
produces much more accurate values.

\begin{figure}
\begin{center}
\includegraphics[width=6cm]{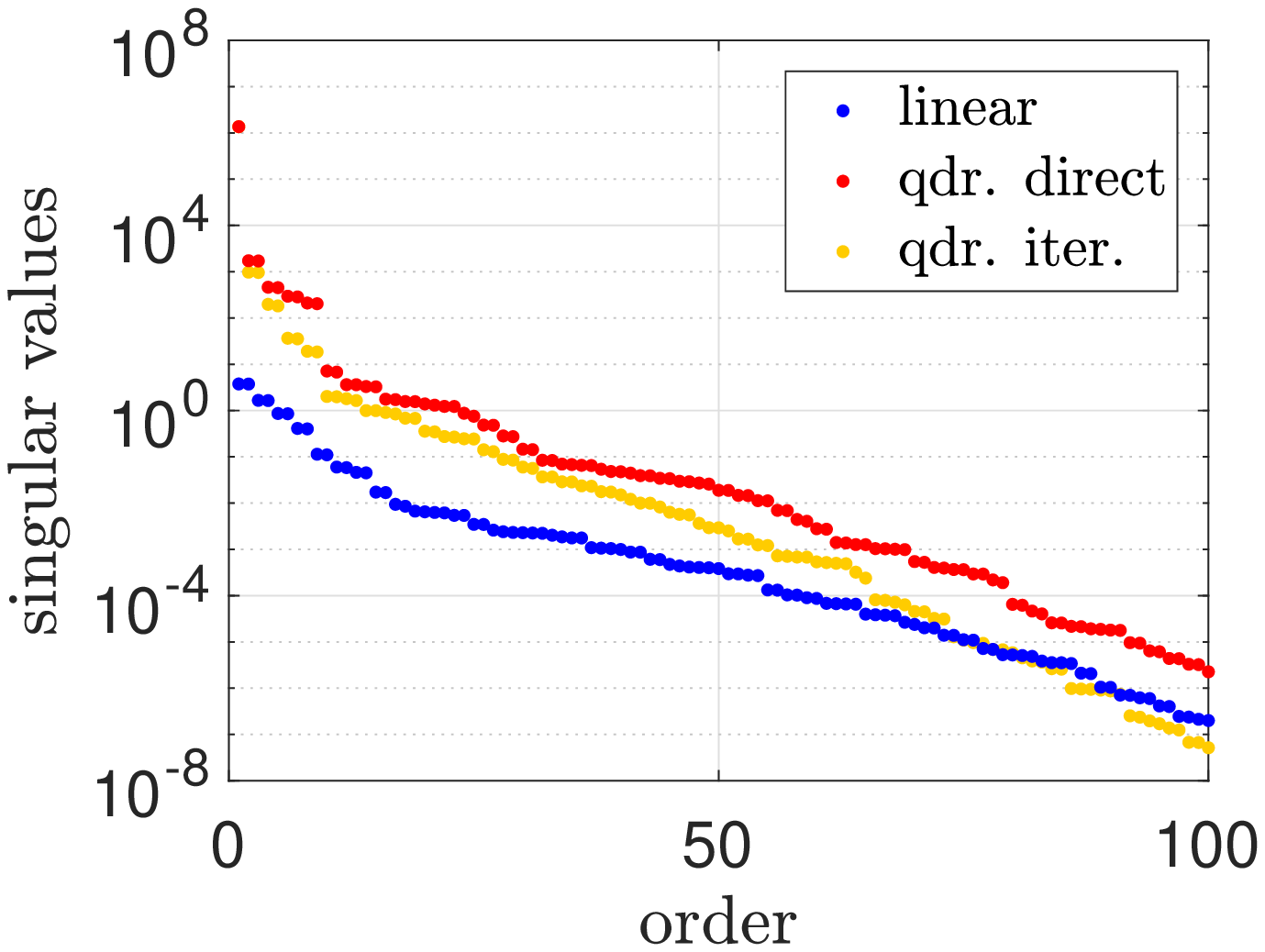}
\hspace{5mm}
\includegraphics[width=6cm]{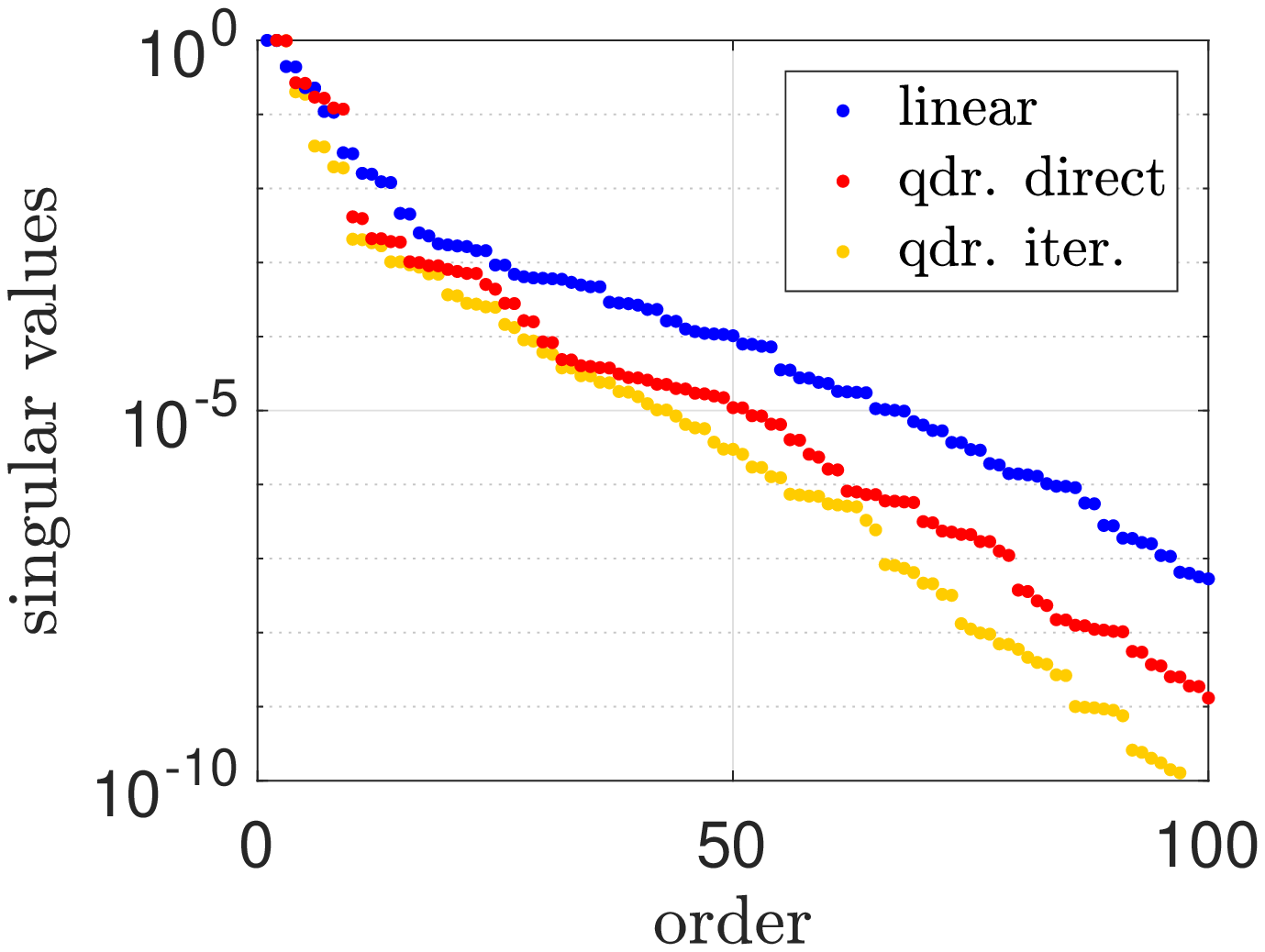}
\end{center}
\caption{Singular values (left) and their normalized values (right)
  for the two dynamical systems in mass-spring-damper example.}
\label{fig:msd-sv}
\end{figure}
 
Given the projection matrices with $r_{\max}$ columns,
we choose the dominant part to obtain ROMs of dimension
$r=5,6,\ldots,100$.
We arrange $\varepsilon = 0$ in the matrix~(\ref{barA}) again.
To investigate the errors of the MOR, we solve initial value problems
highly accurate in the time interval $[0,T]$ by the trapezoidal rule
with constant time step size using $5 \cdot 10^5$ time steps. 
The constant step size allows for reusing $LU$-de\-composi\-tions in
all systems.
The original system~(\ref{dynamicalsystem}) yields the reference solution.
Due to Theorem~\ref{thm:rom}, nonlinear systems of algebraic equations
are omitted in the quadratic ROMs~(\ref{system-reduced2}).
We solve the ROMs for each dimension~$r$.
The maximum absolute errors and the integral mean values of the
relative errors are depicted in Figure~\ref{fig:msd-error}.
The errors decrease exponentially in each approach.
The absolute errors decay exponentially until reduced dimension
$r \approx 90$, and stagnate thereafter.
Our tests suggest that this stagnation is due to the accuracy of the
time integration method.
Including more time steps in the integration routine would remove this
stagnation.
Furthermore, the iteration technique produces approximations of the
same quality as the direct approach.
The relative error is very large for low dimensions in all approaches,
because the exact values of the output are close to zero for small times.

\begin{figure}
  \begin{center}
    \includegraphics[width=6cm]{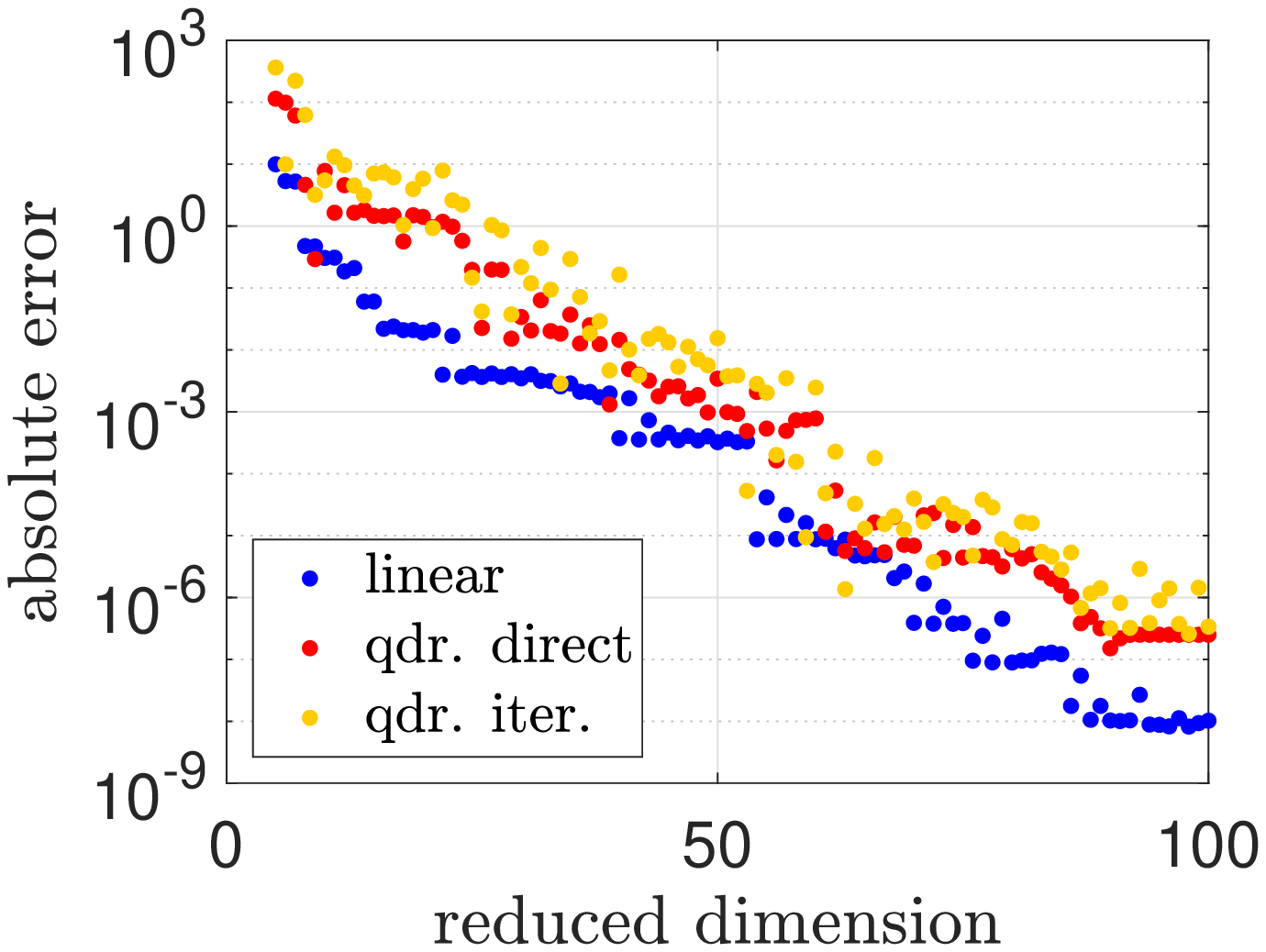}
    \hspace{5mm}
    \includegraphics[width=6cm]{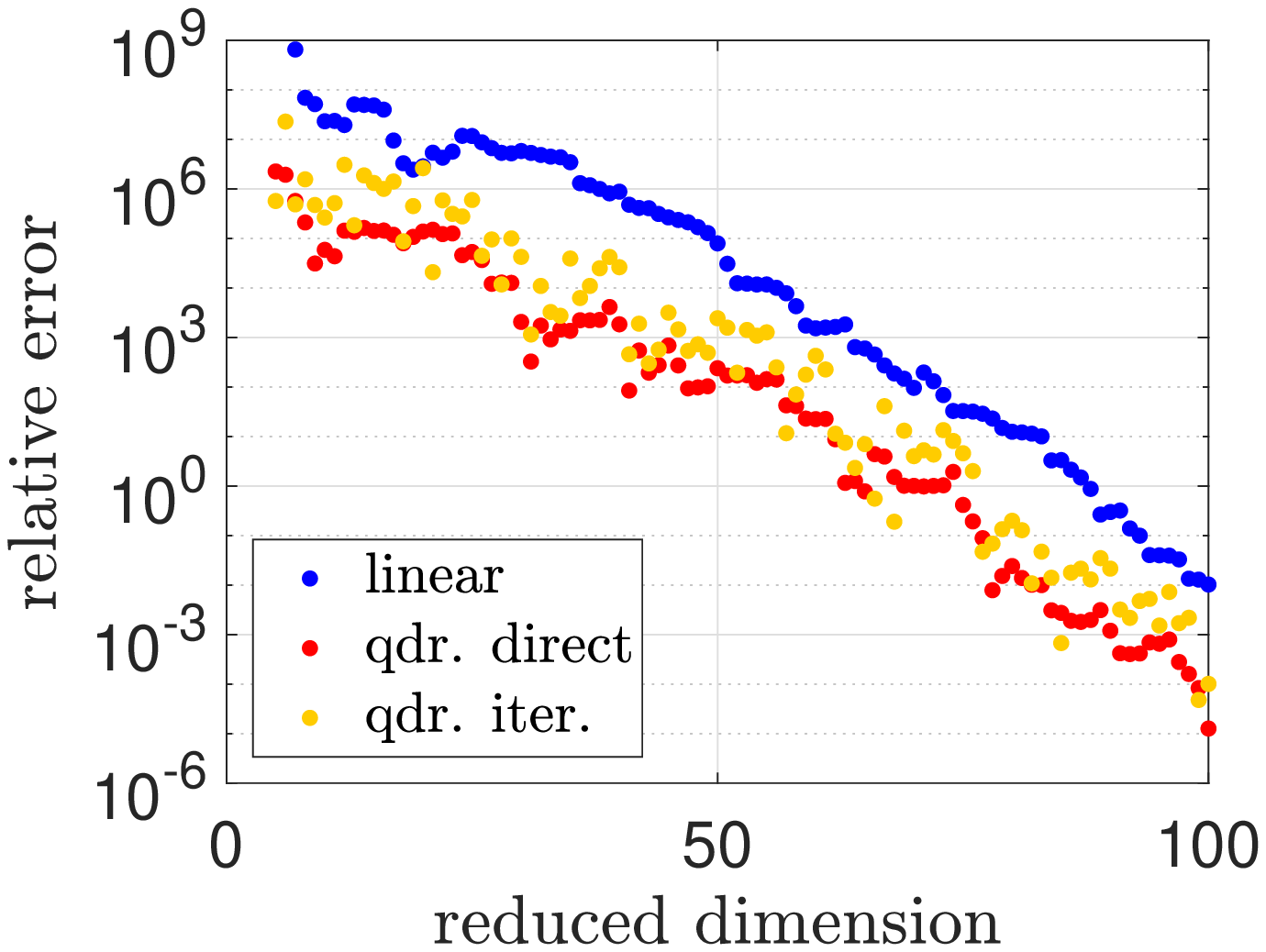}
  \end{center}
  \caption{
    Maximum absolute errors (left) and mean relative errors (right)
    for ROMs of different dimensions in mass-spring-damper example.}
  \label{fig:msd-error}
\end{figure}


\newpage

\section{Conclusions}
We have investigated two approaches for model order reduction of linear dynamical systems with an output of interest that is quadratic in the state variables. This problem can be recast, equivalently, as a linear dynamical system with multiple outputs, or as a quadratic-bilinear system with a single output. Our model order reduction approaches implement the method of balanced truncation for each of these two recast systems. Balanced truncation requires solutions to Lyapunov equations. We find that manipulation of large matrices necessary to solve the Lyapunov equations motivate the use of approximate or iterative approaches, notably the alternating direction implicit method, in both linear and quadratic systems.

Our numerical examples demonstrate that both model order reduction approaches can achieve significant accuracy with a much smaller dynamical system.
For computing the output quantities of interest, the quadratic-bilinear systems are advantageous because they require computation of only a single scalar output.
In contrast, computation of the output quantity of interest from linear dynamical systems requires to compute multiple outputs.
The alternating direction implicit method is not possible in the case of a large number of outputs.
Alternatively, our numerical computations show that the iterative solution is both feasible and faster than a direct solution in the case of the quadratic-bilinear system.
We suppose that a further tuning of the iteration settings may still improve
the efficiency of the used technique, which takes further investigations.


\end{document}